\documentclass{amsart}
\usepackage{fullpage}
\usepackage{pictex}
\usepackage{amsmath}
\usepackage{amssymb}
\usepackage{amsfonts}
\usepackage{amsopn}
\usepackage{epsfig}
\usepackage{amsfonts}
\usepackage{latexsym}
\usepackage{bm}
\usepackage{stmaryrd}
\usepackage{amsthm}
\usepackage{mathrsfs}
\usepackage{stdclsdv}

\newtheorem{theorem}{Theorem}[section]
\newtheorem{lemma}[theorem]{Lemma}
\newtheorem{proposition}[theorem]{Proposition}

\newtheorem{remark}{Remark}
\newtheorem{example}{Example}

\newcommand{\R}{\mathbb R} \newcommand{\Z}{\mathbb Z}
\newcommand{\N}{\mathbb{N}} \newcommand{\Q}{\mathbb{Q}}

\usepackage{tikz}

\title{Sharpening Borel's result in Diophantine approximation}

\author{Savita S.\ Adhin, Ayreena Bakhtawar \and Cor Kraaikamp}
\newcommand{\Addresses}{{% additional braces for segregating \footnotesize
  \bigskip
  \footnotesize

Savita S.\ Adhin, \textsc{
Delft University of Technology, EWI (DIAM), Mekelweg 4, 2628 CD Delft, the Netherlands} 
\par\nopagebreak
\textit{E-mail address}: \texttt{sadhin@tudelft.nl}

A.~Bakhtawar, \textsc{
Institute of Mathematics, Polish Academy of Sciences, ul.  Sniadeckich 8, 00-656
Warszawa, Poland}\par\nopagebreak
  \textit{E-mail address}: \texttt{abakhtawar@impan.pl; ayreena.bakhtawar@gmail.com}

  \medskip

C.~Kraaikamp\textsc{
Delft University of Technology, EWI (DIAM), Mekelweg 4, 2628 CD Delft, the Netherlands} 
\par\nopagebreak
\textit{E-mail address}: \texttt{c.kraaikamp@tudelft.nl}
}
}

\begin{document}
%\maketitle

\begin{abstract}  %\noindent 
In this paper, we refine Borel's 1903 result in Diophantine approximation by providing sharper bounds for the minimum of three consecutive approximation coefficients $\Theta_n(x)$, defined for any real number $x$ with regular continued fraction (RCF) expansion $x=[0;a_1,a_2,\dots]$ as $\Theta_n = q_n^2\left| x-\frac{p_n}{q_n}\right|$. Here $\frac{p_n}{q_n}$ is the $n$th RCF convergent of $x$. Borel's result states that for all (irrational) $x$ and all $n\in\N$,
$$
\min \left\{ \Theta_{n-1}(x),\Theta_n(x),\Theta_{n+1}(x)\right\} \leq \frac{1}{\sqrt{5}}.
$$
We focus on the situation where $a_{n+1}=1$, since otherwise a result by F.~Bagemihl and J.R.~McLaughlin from 1966 implies that the Borel-bound $1/\sqrt{5}$ can already be improver to $1/\sqrt{8}$. 
\end{abstract}

\maketitle

\section{Introduction}\label{section1}

Two classic results in Diophantine approximation are Vahlen's well-known theorem of 1895, \cite{[V]}, and Borel's 1903 theorem, \cite{[Bor1],[Bor2]}. Let\footnote{Without loss of generality we will assume in this paper that $x\in [0,1)$.} $x\in\R$ have \emph{regular continued fraction} (RCF) expansion 
\begin{equation}\label{RCFexpansion}
x = a_0 + \frac{1}{a_1+ \frac{\displaystyle 1}{\displaystyle a_2+\ddots +\frac{1}{a_n+\ddots}}},
\end{equation}
which is denoted by $x=[a_0;a_1,a_2,\dots]$, where $a_0\in\Z$, such that $x-a_0\in [0,1)$, and\footnote{Recall that the RCF-expansion of $x\in\R$ is finite if and only if $x\in\Q$. In that case $x$ has two RCF-expansions. In all other cases the RCF-expansion of $x\in\R$ is unique.}  $a_i\in\N$, for $n\geq 1$. Furthermore, by finite truncation in~\eqref{RCFexpansion}, we get the sequence of \emph{regular continued fraction convergents} $(p_n/q_n)$ of $x$, given by
\begin{equation}\label{RCFconvergents}
\frac{p_n}{q_n} = a_0 + \frac{1}{a_1+ \frac{\displaystyle 1}{\displaystyle a_2+\ddots +\frac{1}{a_n}}} = [a_0;a_1,\dots,a_n],
\end{equation}
where $p_n,q_n\in\Z$, $q_n>0$, and $\text{gcd}\{ p_n,q_n\}=1$, for $n\geq 1$. Then~\eqref{RCFexpansion} expresses that
$$
\lim_{n\to\infty} \frac{p_n}{q_n} = x.
$$
For a proof and other information on continued fractions used in this paper, see, e.g.~\cite{[BvdPSZ], [DK], [Hen], [IK], [RS]}, but also the classic but still highly relevant books~\cite{[K],[P]}.

Define the \emph{approximation coefficients} $\Theta_n=\Theta_n(x)$ of $x\in\R$ by
$$
\Theta_n(x) = q_n^2\left| x-\frac{p_n}{q_n}\right| ,\quad \text{for $n\geq 1$},
$$
then Vahlen's theorem states that for $x\in\R$ and $n\geq 1$, we have
\begin{equation}\label{Vahlen}
\min \left\{ \Theta_{n-1}(x) ,\,\, \Theta_n(x) \right\} < \tfrac{1}{2}.
\end{equation}
In 1903, \'Emile Borel (\cite{[Bor1], [Bor2]}) showed, that if $x\in\R$, with three consecutive convergents $\tfrac{p_{n-1}}{q_{n-1}}$, $\tfrac{p_n}{q_n}$
and $\tfrac{p_{n+1}}{q_{n+1}}$, one has
\begin{equation}\label{Borel}
\min \left\{ \Theta_{n-1}(x),\Theta_n(x),\Theta_{n+1}(x)\right\} < \frac{1}{\sqrt{5}}.
\end{equation}
As $0\leq \Theta_n(x)<1$, for $n\geq 1$, and since the sequence of denominators $(q_n)_{n\geq}$ is growing monotonically and exponentially,~\eqref{Vahlen} and~\eqref{Borel}  show that RCF-algorithm yields very good rational approximations to (irrational, but also rational) numbers $x$. This is further underlined by a result of Legendre from 1798 (\cite{[L],[BJ]}). Let $A,B\in\Z$ with $B>0$, and $\text{gcd}\{ A,B\}=1$ be such that 
\begin{equation}\label{legendre}
B^2\left| x-\frac{A}{B}\right| < \frac{1}{2}.
\end{equation}
The rational number $A/B$ is an RCF convergent of $x$; there exists an $n\in\N$, such that $A=p_n$ and $B=q_n$. So \emph{if} we want to approximate an irrational number $x$ well by a rational number $A/B$ (such that~\eqref{legendre} holds), Legendre's result states that this rational number is an RCF-convergent of $x$.\smallskip\

Recently, Vahlen's result was sharpened by Jaroslav Han\u{c}l in~\cite{[H]}, by Han\u{c}l and Silvie Bahnerova in~\cite{[HB]}, and by Dinesh Sharma Bhattarai in 2023 (\cite{[B]}). Taking an approach very different from~\cite{[B], [H], [HB]}, Vahlen was extensively sharpened by the second and third authors in~\cite{[BK]}.\smallskip

More recently, Borel's result was sharpened by Han\u{c}l and Radhakrishnan Nair in~\cite{[HN]} who showed that if we replace the $<$ sign in~\eqref{Borel} by $\leq$,  the constant $\sqrt{5}$ can be replaced by $\sqrt{5}+\frac{4-5\sqrt{5}+\sqrt{61}}{2q^2}$. In this paper, we also will sharpen Borel's result. Apart from when $x$ equals the (small) golden mean $g$, where\footnote{Note that $gG=1$, $g^2=1-g$, $G^2=1+G$, and $1/g=G$.}
\begin{equation}\label{goldenmean}
g=\frac{\sqrt{5}-1}{2},\quad \text{and}\quad G=\frac{\sqrt{5}+1}{2},
\end{equation}
the approach will be different from that of Han\u{c}l and Nair, and more along the lines of~\cite{[BK]}. Not only is the approach we follow very different from that of Han\u{c}l and Nair, also the results are different. Han\u{c}l and Nair give a bound which depends only on the denominator $q$ of one of the three convergents $\tfrac{p_{n-1}}{q_{n-1}}$, $\tfrac{p_n}{q_n}$, $\tfrac{p_{n+1}}{q_{n+1}}$ involved, and can be seen in some way as universal. However, when the partial quotient $a_{n+1} > 1$, their result is quite conservative, certainly for large $n$. In case $a_{n+1}=1$, we consider other partial quotients in order to improve upon Borel's result.\smallskip\

Before stating our results in Section~\ref{section3} and and proving our results in Section~\ref{section:combining}, in Sections~\ref{section2} and~\ref{section3} we introduce some tools already used in~\cite{[BK]}, but now these tools will be used in a more refined way. We will also review older refinements of Borel by M.~Fujiwara (\cite{[F]}), N.~Obrechkoff (\cite{[O]}), A.L.~Schmidt (\cite{[S]}), and F.~Bagemihl \& J.R.~McLaughlin (\cite{[BMc]}), and the ``converse property'' by Jingcheng Tong (\cite{[T]}).

\section{Using the Nakada natural extension of the RCF}\label{section2}
In 1981, Hitoshi Nakada (in~\cite{[N]}) obtained the \emph{natural extension} of the dynamical systems underlying his $\alpha$-expansions for all $\alpha\in [\tfrac{1}{2},1]$ (which also includes the RCF, which is $\alpha =1$), derived the invariant measures, and showed that these dynamical systems are ergodic. Nakada's natural extensions for the RCF played a major role in the development of the metrical theory of continued fractions since the early 1980ies. See Chapter~4 in~\cite{[IK]}.\smallskip\

\subsection{Nakada's natural extension for the RCF}\label{subsectionNakada} Let $T:[0,1)\to [0,1)$ be the so-called \emph{Gauss map}, defined by $T(0)=0$, and
$$
T(x) = \frac{1}{x}-\left\lfloor \frac{1}{x}\right\rfloor ,\quad \text{for $x\in (0,1)$}.
$$
Then one of the possible versions of Nakada's natural extension map $\mathcal{T}$ is the map $\mathcal{T}:[0,1)\times [0,1]\to [0,1]^2$, defined by
\begin{equation}\label{naturalextensionmap}
\mathcal{T}(x,y) = \left( T(x),\frac{1}{\left\lfloor \tfrac{1}{x}\right\rfloor + y}\right),\quad \text{if $(x,y)\in (0,1)\times [0,1]$},
\end{equation}
and $\mathcal{T}(0,y) = (0,y)$, for $y\in [0,1]$. Nakada showed in 1981 in~\cite{[N]} that $(\Omega,\mathcal{B},\bar{\mu},\mathcal{T})$ forms an ergodic system, where $\mathcal{B}$ is the collection of Borel sets of $\Omega$, and $\bar{\mu}$ is a $\mathcal{T}$-invariant probability measure, with density $d(x,y)$, where
$$
d(x,y) = \frac{1}{\log 2} \frac{1}{(1+xy)^2},\quad \text{for $(x,y)\in\Omega$},
$$
and $d(x,y)=0$ elsewhere. In fact, in~\cite{[N]} Nakada obtained stronger mixing properties than ergodicity of the dynamical system $(\Omega,\mathcal{B},\bar{\mu},\mathcal{T})$. However, for our purposes, ergodicity suffices.\medskip\

For $x\in (0,1)$ and $n\geq 1$, with RCF-convergents $(\tfrac{p_n}{q_n})$,  we define 
\begin{equation}\label{future}
t_n=t_n(x) := T^n(x),
\end{equation}
as the \emph{future of $x$ at time} $n$, and 
\begin{equation}\label{past}
v_n=v_n(x) := \frac{q_{n-1}}{q_n},
\end{equation}
as the \emph{past of $x$ at time} $n$. Note that if $x=[0;a_1,a_2,\dots]$, we have $t_n=[0;a_{n+1},a_{n+2},\dots]$, and from the well-known recurrence relations for the $q_n$ we furthermore have $v_n=[0;a_n,a_{n-1},\dots,a_1]$. See pages 7--10 in~\cite{[Sch]}. It follows from the definition of $\mathcal{T}$ in~\eqref{naturalextensionmap} that
$$
\mathcal{T}^n(x,0) = (t_n,v_n),\quad \text{for $n\geq 0$}.
$$
\begin{remark}\label{RemarkOnJager}{\rm
For the optimality of the improved Borel-bounds $B_{\ell,h}$ we will introduce in Section~\ref{section3} and study in Section~\ref{section:combining}, it is important to remark that Henk Jager showed in ~\cite{[J]} that for Lebesgue almost all $x$, the sequence $(t_n,v_n)_{n\geq 0}$ behaves asymptotically like a generic sequence. I.e., for Lebesgue a.e.\ $x$, the sequence $(t_n,v_n)_{n\geq 0}$ will visit any Borel-set $A\subset\Omega$ with asymptotic frequency $\bar{\mu}(A)$. See also~\cite{[BJW]}, where this result, although not explicitly stated as such, formed one of the main ingredients of the proof of the so-called \emph{Doeblin-Lenstra conjecture}. See also~\cite{[DK],[IK]} for more details.}\hfill$\triangle$
\end{remark}

A classical result is now, that
\begin{equation}\label{thetan}
\Theta_n(x) = \frac{t_n}{1+t_nv_n},\quad \text{for $n\geq 1$}.
\end{equation}
For a proof of~\eqref{thetan}, see~\cite{[DK], [IK], [K]}, but also (2) in~\cite{[B]} and (3.2) in~\cite{[H]}. In fact, from~\eqref{thetan} and the definition of the Gauss map $T$, it is easily found that\footnote{From now on we will usually suppress the dependence of $\Theta_n(x)$ on $x$.}
\begin{equation}\label{thetan-1}
\Theta_{n-1}(x) = \frac{v_n}{1+t_nv_n},\quad \text{for $n\geq 1$}.
\end{equation}
In view of~\eqref{thetan} and~\eqref{thetan-1}, Henk Jager and the third author introduced and studied in~\cite{[JK]} the map $\Psi: \Omega := [0,1]\times [0,1]\to \R^2$, defined by
\begin{equation}\label{PsiMap}
\Psi (t,v) = \left( \frac{v}{1+tv},\frac{t}{1+tv}\right),\quad \text{for $(t,v)\in\Omega$}.
\end{equation}
In~\cite{[JK]}, it is shown that $\Psi (\Omega)$ is the  triangle $\Delta$ in $\R^2$ with vertices $(0,0)$, $(0,1)$, and $(1,0)$. A trivial but handy observation from~\eqref{PsiMap} is that if $(t,v)\in\Omega$ is above/on/under the diagonal $v=t$, the image $\Psi (t,v)$ of $(t,v)$ under $\Psi$ is under/on/above the diagonal $\beta=\alpha$ in $\Delta$. Since $(\Theta_{n-1},\Theta_n) = \Psi (t_n,v_n)\in\Delta$, for $n\geq 1$, we find that $\Theta_{n-1} + \Theta_n < 1$, which implies Vahlen's result 
$$
\min \{ \Theta_{n-1},\Theta_n\} < \tfrac{1}{2}.
$$
%\smallskip\

Furthermore, off a set of Lebesgue measure zero, the map $\Psi :\Omega\to\Delta$ is bijective. If we define the map $F:\Delta\to\Delta$ by 
\begin{equation}\label{Fmap}
F=\Psi\circ \mathcal{T}\circ \Psi^{-1},
\end{equation}
we find for two consecutive $\Theta_{n-1}$ and $\Theta_n$ that
$$
F(\Theta_{n-1},\Theta_n) = \left( \Theta_n,\Theta_{n+1}\right) ,\quad \text{for $n\geq 1$}.
$$
As a result, one can derive the formula of  W.B.~Jurkat and A.~Peyerimhoff (\cite{[JP]})
\begin{equation}\label{thetasA}
\Theta_{n+1} = \Theta_{n-1} +a_{n+1}\sqrt{1-4\Theta_{n-1}\Theta_n} - a_{n+1}^2\Theta_n.
\end{equation}
See also~\cite{[JK]}, where~\eqref{thetasA} was used to obtain, in a simple way, a generalization of Borel's celebrated result~(\ref{Borel}). This generalization of Borel's result was previously obtained among others by M.~Fujiwara (\cite{[F]}), N.~Obrechkoff (\cite{[O]}), A.L.~Schmidt (\cite{[S]}), and by F.~Bagemihl and J.R.\ McLaughlin (\cite{[BMc]})
\begin{equation}\label{improvementBorel}
\min \{ \Theta_{n-1},\Theta_n,\Theta_{n+1}\} \leq \frac{1}{\sqrt{a_{n+1}^2+4}},\quad \text{for $n\geq 1$}.
\end{equation}
As $\frac{1}{\sqrt{a_{n+1}^2+4}}\leq \frac{1}{\sqrt{5}}$, \eqref{improvementBorel} immediately implies Borel's result~\eqref{Borel}. In 1983, Jingcheng Tong obtained in~\cite{[T]} the ``conjugate property'' of~(\ref{improvementBorel})
\begin{equation}\label{Tong}
\max \{ \Theta_{n-1},\Theta_n,\Theta_{n+1}\} \geq \frac{1}{\sqrt{a_{n+1}^2+4}},\quad \text{for $n\geq 1$}.
\end{equation}

\subsection{Easy proofs of~\eqref{improvementBorel} and~\eqref{Tong}}\label{subsec:valensconstants} In this section, we briefly recall the proofs of~\eqref{improvementBorel} and~\eqref{Tong} from~\cite{[JK]}, as the approach from~\cite{[JK]} is extensively used in this paper. In view of~\eqref{thetan} and~\eqref{thetan-1}, we define for $a\in\N$ subsets $V_a$ and $H_a$ of $\Omega$ as follows
$$
V_a = \left\{ (x,y)\in\Omega\,\, \Big{|}\,\, \frac{1}{a+1} < x\leq \frac{1}{a}\right\} ,\quad \text{and}\quad 
H_a = \left\{ (x,y)\in\Omega\,\, \Big{|}\,\, \frac{1}{a+1} < y\leq \frac{1}{a}\right\}.
$$
One can easily show that $\mathcal{T}(V_a)=H_a$ for all $a\in\N$. Note that the following conditions hold
\begin{eqnarray*}
\mathcal{T}^n(t,v)\in V_a &\Leftrightarrow & a_{n+1}=a,\,\, n\geq 0
\end{eqnarray*}
and
\begin{eqnarray*}
\mathcal{T}^n(t,v)\in H_a &\Leftrightarrow & a_n=a,\,\, n\geq 1.
\end{eqnarray*}
See also Figure~\ref{figure1}, where the sets $V_a$ and their image under $\Psi$ are given for $a=1,2,3$.

\begin{remark}\label{remarkonhorizontalandverticallines}{\rm
Note that if $c\in [0,1]$ is a constant, it is easily seen that $\Psi$ maps the vertical line segment $t=c$, $v\in [0,1]$, to the intersection of the line $\beta = -c^2\alpha +c$ with $\Delta$. Also it is easily seen that $\Psi$ maps the horizontal line segment $t\in [0,1]$, $v=d\in (0,1]$ fixed, to the intersection of the line $\beta = -\tfrac{1}{d^2}\alpha +\tfrac{1}{d}$ with the triangle $\Delta$.}\hfill$\triangle$
\end{remark}

Due to Remark~\ref{remarkonhorizontalandverticallines}, for $a\geq 2$ we find that $\Psi (V_a)$ is a quadrilateral in $\Delta$, with vertices 
$$
\Psi \left(\tfrac{1}{a},0\right) = \left( 0,\tfrac{1}{a}\right),\quad \Psi \left( \tfrac{1}{a+1},0\right) = \left( 0,\tfrac{1}{a+1}\right),\quad \Psi \left( \tfrac{1}{a+1},1\right) = \left( \tfrac{a+1}{a+2},\tfrac{1}{a+2}\right),
$$
and ${\displaystyle \Psi \left( \tfrac{1}{a},1\right) = \left( \tfrac{a}{a+1},\tfrac{1}{a+1}\right)}$. For $a=1$ we have that $\Psi (V_a)$ is a triangle in $\Delta$, with vertices 
$$
\left( 0,\tfrac{1}{2}\right),\quad \left( \tfrac{2}{3},\tfrac{1}{3}\right),\quad \text{and $(0,1)$}.
$$
\begin{figure}[h]
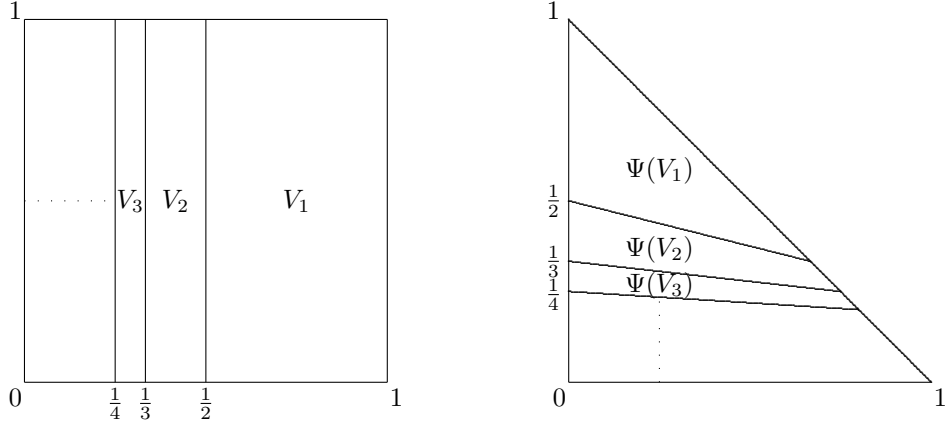

$$
\beginpicture
    \setcoordinatesystem units <0.4cm,0.4cm>
    \setplotarea x from -15 to 15, y from 0 to 12
    \putrule from 3 0 to 3 12
    \putrule from 3 0 to 15 0
    \put {$0$} at 2.7 -0.5
    \put {$0$} at -15.3 -0.5
    \put {$1$} at 15.3 -0.5
    \put {$1$} at -2.7 -0.5
    \put {$1$} at 2.5 12.3
    \put {$1$} at -15.3 12.3
    \putrule from -15 0 to -3 0
    \putrule from -15 0 to -15 12
    \putrule from -3 0 to -3 12
    \putrule from -15 12 to -3 12
    \putrule from -9 0 to -9 12
\putrule from -11 0 to -11 12
\putrule from -12 0 to -12 12

    \put {$\tfrac{1}{2}$} at 2.5 6
    \put {$\tfrac{1}{2}$} at -9 -0.7
    \put {$\tfrac{1}{3}$} at 2.5 4
    \put {$\tfrac{1}{3}$} at -11 -0.7
    \put {$\tfrac{1}{4}$} at 2.5 2.9
    \put {$\tfrac{1}{4}$} at -12 -0.7
    \put {$\Psi(V_1)$} at 6 7
    \put {$\Psi(V_2)$} at 6 4.4
    \put {$\Psi(V_3)$} at 6 3.2
    \put {$V_1$} at -6 6
    \put {$V_2$} at -10 6
    \put {$V_3$} at -11.5 6 

\setlinear \plot
3 12 15 0
/
\setlinear \plot
3 6 11 4
/
\setlinear \plot
3 4 12 3
/
\setlinear \plot
3 3 12.6 2.4
/

\setdots
\putrule from 6 0 to 6 3
\putrule from -15 6 to -12 6

\endpicture
$$ \caption[triangle]{The natural extension $\Omega$ of the RCF (left), and its image $\Delta$ under $\Psi$ (right)}\label{figure1}
\end{figure}
Thus $\Psi (V_a)$ is bounded by the lines: $\alpha = 0$, $\alpha+\beta =1$, $\beta = -\frac{\alpha}{a^2}+\frac{1}{a}$ (this is the `top-line' of $\Psi (V_a)$), and $\beta = -\frac{\alpha}{(a+1)^2}+\frac{1}{a+1}$ (this is the `bottom-line' of $\Psi (V_a)$, and the `top-line' of $\Psi (V_{a+1})$); see Figure~\ref{figure1}. Also note, that $\Psi (H_a)$ is the reflection of $\Psi (V_a)$ in the line $\beta = \alpha$. So $\Psi (H_a)$ is bounded by the lines $\beta =0$, $\alpha +\beta =1$, the ``right-hand line'' $\beta = -a^2\alpha + a$, and ``left-hand line'' $\beta = -(a+1)^2\alpha+(a+1)$.\medskip\

The essence of the proof of~\eqref{improvementBorel} and~\eqref{Tong} in~\cite{[JK]} is that for each $a\in\N$ there is a unique quadratic irrational number $\xi_a\in \big( \tfrac{1}{a+1},\tfrac{1}{a}\big]$, which is a fixed point\footnote{The bar indicates the period, which is for $\xi_a$ of length 1.} of the Gauss-map $T$: 
$$
\xi_a = [0; \overline{a}] = \frac{-a+\sqrt{a^2+4}}{2}.
$$
Consequently, by definition of $\mathcal{T}$ we see that $(\xi_a,\xi_a)$ is a fixed point of $\mathcal{T}$, and since
$$
\Psi \left( \xi_a,\xi_a \right) = \left( \frac{1}{\sqrt{a^2+4}}, \frac{1}{\sqrt{a^2+4} } \right),\quad \text{for all $a\in\N$},
$$
setting
\begin{equation}\label{eta_a}
\eta_a= \frac{1}{\sqrt{a^2+4}},\quad \text{for all $a\in\N$},
\end{equation}
we see that by definition~\eqref{Fmap} of $F$ the point $(\eta_a,\eta_a)$ is a fixed point of $F$.\smallskip\

In view of~\eqref{thetasA}, we now consider on each $V_a^*=\Psi (V_a)$, for $a\in\N$, the map $f:V_a^*\to [0,1]$, defined by
\begin{equation}\label{fmap}
f(\alpha,\beta) = \alpha + a\sqrt{1-4\alpha\beta} - a^2\beta,\quad \text{for $(\alpha,\beta)\in V_a*$}.
\end{equation}
One can show that
\begin{equation}\label{derivatesarenegative}
\frac{\partial}{\partial \alpha} f(\alpha,\beta) < 0,\quad \text{and that}\quad \frac{\partial}{\partial \beta} f(\alpha,\beta) < 0
\end{equation}
(the last one is trivial, the first one is some work). From this and the fact that
$$
f(\eta_a,\eta_a) = f\left( \frac{1}{\sqrt{a^2+4}}, \frac{1}{\sqrt{a^2+4}}\right) = \frac{1}{\sqrt{a^2+4}},
$$
we immediately find~\eqref{improvementBorel} and~\eqref{Tong}. For more details, see~\cite{[JK]}. 

\section{Generalizing Borel's result}\label{section3}
It follows from~\eqref{improvementBorel} that whenever $a_{n+1}\geq 2$, one already has a considerable improvement over Borel's result~\eqref{Borel}, and this was also noted in~\cite{[HN]}. In that case the constant $1/\sqrt{5}=0.447213595\dots$ is replaced by a constant at most $1/\sqrt{8}=0.35355339\cdots$. Therefore, in order to improve upon Borel's result, we will assume in the rest of this paper that $a_{n+1}=1$. We will also assume that $a_n=1$. If $a_{n+1}=1$, and $a_n\geq 2$, then $(t_n,v_n)\in [\tfrac{1}{2},1]\times [0,\tfrac{1}{2}]$, which is mapped by $\Psi$ to the quadrilateral in $\Delta$ with vertices $(0,\tfrac{1}{2})$, $(0,1)$, $(\tfrac{1}{3},\tfrac{2}{3})$, and $(\tfrac{2}{5},\tfrac{2}{5})$. As the last point is the bottom right vertex of $\Psi([\tfrac{1}{2},1]\times [0,\tfrac{1}{2}])$, it is seen that for points $(\alpha,\beta)$ in this quadrilateral we have $\alpha\leq \tfrac{2}{5}$, and $\beta\geq \tfrac{2}{5}$. So for any $(\alpha,\beta)\in \Psi([\tfrac{1}{2},1]\times [0,\tfrac{1}{2}])$, whatever the value of $f(\alpha,\beta)$, we have $\min \left\{ \alpha,\beta,f(\alpha,\beta)\right\} \leq \tfrac{2}{5}$, which is a considerable improvement of Borel's constant $\tfrac{1}{\sqrt{5}}$. In fact, $f(\tfrac{2}{5},\tfrac{2}{5}) = \tfrac{3}{5}$, so $\tfrac{2}{5}$ cannot be replaced in $\Psi([\tfrac{1}{2},1]\times [0,\tfrac{1}{2}])$ by a smaller constant on the quadrilateral under consideration.\medskip

\subsection{Main result}\label{subsec:MainResult}
In the next section, we first assume that $t_n\neq g$. Later, in Section~\ref{section5}, the case $t_n=g$ will be treated as a special case. Gathering all the various results together, we will obtain the following theorem, which is the main result of this paper.
\begin{theorem}\label{thm:MainResult}
Let $x\in [0,1)$, $x=[0;a_1,a_2,\dots]$, and for $n\in\N$, let $t_n=T^n(x)=[0;a_{n+1},a_{n+2},\dots]$, and $v_n=\frac{q_{n-1}}{q_n}=[0;a_n,a_{n-1},\dots,a_1]$; see~\eqref{future} resp.~\eqref{past}. Furthermore, assuming that $a_{n+1}=1=a_n$, we have the following cases, where $\ell\in\N$ and $1\leq h\leq n$, 
\begin{enumerate}
\item[($i$)]
Let $t_n\neq g$, $t_n=[0;1^{\ell},a_{n+\ell+1}\neq 1,\dots]$ and $v_n=[0;1^h, a_{n-h},\dots,a_1]$ if $1\leq h<n$, and $v_n=[0;1^h]$ if $n=h$. Then
$$
\min \{ \Theta_{n-1}(x), \Theta_n(x), \Theta_{n+1}(x)\} \leq B_{\ell,h} < \frac{1}{\sqrt{5}},
$$
where the value of the Borel-constant $B_{\ell,h}$ is given in Propositions~\ref{ell=h odd},~\ref{ell>hgeq1odd},~\ref{prop:ellhoddh>ell},~\ref{prop:ellhoddh=ell+1},~\ref{prop:ellhoddheven>ell+3},~\ref{prop:ellhoddheven<ell-1},~\ref{prop:ellevenhodd}, \ref{prop:ellhbothevenh=ell}, and~\ref{prop:ellhbothevenh>ell}, depending on the four cases, whether $\ell$ and $h$ are odd or even. The Borel-constants $B_{\ell,h}$ are always best possible, i.e., they cannot be replaced by a smaller constant.
\item[]
\item[($ii$)]
Let $t_n=g$, for some $n\in\N\cup\{ 0\}$, then we have the following cases.\smallskip\
\begin{enumerate}
\item[($iia$)]
Let $n\geq 1$, $a_n\geq 2$, then
$$
\min \{ \Theta_{n-1}, \Theta_n, \Theta_{n+1}\} = \Theta_{n-1} < g^2 < \Theta_{n+1} < \tfrac{1}{\sqrt{5}} < \Theta_n.
$$
\item[]
\item[($iib$)]
Let $n\geq 1$, $1\leq \ell < n$, such that $v_n= [0;1^{\ell}, a_{n-\ell}\neq 1,\dots,a_1]$, and set\footnote{We abbreviate $[0; \underbrace{1,1,\dots,1}_{\ell+2 -\text{times}}]$ by $[0;1^{\ell+2}]$.} $c_{\ell}=[0;1^{\ell+2}]$. Then for $n\in\N$,
$$
\min \{ \Theta_{n-1}, \Theta_n, \Theta_{n+1}\} = \begin{cases}    
\Theta_n\leq \frac{g}{1+g\cdot c_{\ell}} < \frac{1}{\sqrt{5}}<\Theta_{n+1}<\Theta_{n-1}, & \text{if $n\in\N$ is odd}\\
\Theta_{n-1}\leq \frac{c_{\ell}}{1+g\cdot c_{\ell}} < \Theta_{n+1} < \tfrac{1}{\sqrt{5}} < \Theta_n &  \text{if $n\in\N$ is even.}
\end{cases}
$$

\item[]
\item[($iic$)]
Let $n=0$, so $x=g=t_0=[0;\bar{1}]$ and $v_0=0$, and, moreover, we have $t_m=g$ and $v_m=[0;1^m]$, for all $m\in\N$. Then for $n\in\N$,
$$
\min \{ \Theta_{n-1}, \Theta_n, \Theta_{n+1}\} = \begin{cases}
\Theta_n = \frac{g}{1+g\cdot c_{n-2}} < \frac{1}{\sqrt{5}} < \Theta_{n+1}<\Theta_{n-1}, & \text{if $n$ is \emph{odd}}\\
\Theta_{n-1}  = \frac{g}{1+g\cdot c_{n-3}} < \Theta_{n+1} < \frac{1}{\sqrt{5}} < \Theta_n, & \text{if $n$ is \emph{even}}. 
\end{cases}
$$
\end{enumerate}
\end{enumerate}
\end{theorem}

\begin{remark}\label{RemarkOnBestValue}{\rm
In part ($i$) of our Main Theorem, Theorem~\ref{thm:MainResult}, we state that the various Borel-bounds $B_{\ell,h}$ are always best possible, by which we mean that they cannot be replaced by a smaller constant. We do not mention this explicitly in the proofs of the various propositions, but we show this in every proof by showing that on most of the quadrilaterals (which is occasionally a triangle) under investigation, for one of the variables (i.e., $\alpha$, $\beta$, $f(\alpha,\beta)$, or equivalently: $\Theta_{n-1}$, $\Theta_n$, $\Theta_{n+1}$) the maximum of this variable in the domain under consideration is the improved Borel-bound $B_{\ell,h}$, while for the other variables some values are larger than this bound in the domain under consideration, and we show that on some of the quadrilaterals this is shared by two variables. In fact, we already saw an example of the first kind at the beginning of this section: for points $(\alpha,\beta)$ in the quadrilateral with vertices $(0,\tfrac{1}{2})$, $(0,1)$, $(\tfrac{1}{3}$,$\tfrac{2}{3})$, and $(\frac{2}{5},\tfrac{2}{5})$, we have $\alpha <\tfrac{2}{5}$, and this is already enough to improve the Borel-bound $\tfrac{1}{\sqrt{5}}$. However, since $\beta\geq \tfrac{2}{5}$ on this quadrilateral, and $f(\frac{2}{5},\tfrac{2}{5})=\tfrac{3}{5}>\tfrac{2}{5}$, the bound $\tfrac{2}{5}$ cannot be improved on this quadrilateral under investigation and is therefore our improved Borel-bound. 

It is also important to mention that due to the afore mentioned result of Jager (see Remark~\ref{RemarkOnJager}), for Lebesgue almost all $x$ and any Borel-set $A\subset\Omega$ of positive Lebesgue-measure (or two-dimensional Gauss-measure $\bar{\mu}(A)$) the sequence $(t_n,v_n)_{n\geq 0}$ will visit the set $A$ infinitely often (in fact, due to the Ergodic Theorem, asymptotically $\bar{\mu}(A)$ of the time), and by the map $\Psi$ this carries over to $\Delta$. So we cannot hope to ``deflate'' our quadrilaterals in order to find an even smaller improved Borel-bound $B_{\ell,h}$.
}\hfill $\triangle$
\end{remark}

\subsection{Some more tools}\label{sec:tools}
In~\cite{[BK]}, Vahlen's result~\eqref{Vahlen} was improved by ``staying away'' from the point $\left( \tfrac{1}{2},\tfrac{1}{2}\right)\in\Psi (\Omega)$. The second and third author did this by introducing cylinders $\Delta (k)$, and $\Delta (1,k)$, defined for $k\in\N$ by
$$
\Delta (k) = \left\{ x\in [0,1)\, \Big{|} \, a_1(x)=k\right\},\quad \text{and}\quad
\Delta (1,k) = \left\{ x\in [0,1)\, \Big{|} \, a_1(x)=1, a_2(x)=k\right\}.
$$
An easy calculation yields $\Delta (1)=(\tfrac{1}{2},1)$, $\Delta (1,1)=(\tfrac{1}{2},\tfrac{2}{3})$, $\Delta (k)= \Big(\tfrac{1}{k+1},\tfrac{1}{k}\Big]$, and $\Delta (1,k)= \Big[\tfrac{k}{k+1},\tfrac{k+1}{k+2}\Big)$, for $k\in\N$, $k\geq 2$. Consider for $k\in\N$ the sets $V(1,k)$ and $H(1,k)$, defined by
$$
V(1,k)=\Delta (1,k)\times [0,1]\quad \text{and}\quad H(1,k)=[0,1]\times \Delta (1,k).
$$
Now for every $k,k'\in\N$ one has $\left( \tfrac{1}{2},\tfrac{1}{2}\right)\not\in \Psi(V(1,k))\cup \Psi(H(1,k'))$, and therefore $\left( \tfrac{1}{2},\tfrac{1}{2}\right)\not\in \Psi(V(1,k))\cap \Psi(H(1,k'))$, and from this improvements to Vahlen's result~\eqref{Vahlen} immediately follow; see~\cite{[BK]}. To improve Borel's result~\eqref{Borel} this is \textbf{not} enough. When $x=g=\xi_1$, we have $T(g)=g$, $\Psi (g,g)=(\eta_1,\eta_1)$ (recall that $\eta_1$ is Borel's constant $1/\sqrt{5}$; see~\eqref{Borel} and~\eqref{eta_a}), 
$$
\Psi \left(\mathcal{T}^n(g,0)\right) = \Psi \left(\left( g,\frac{F_{n-1}}{F_n}\right)\right) = (\Theta_{n-1}(g),\Theta(g))\in \Psi (V(1,1)),\quad \text{for $n\geq 0$},
$$
and
$$
\lim_{n\to\infty} \mathcal{T}^n(g,0) = \lim_{n\to\infty} \left( g,\frac{F_{n-1}}{F_n}\right) = (g,g)\in V(1,1)\cap H(1,1).
$$
Here $(F_n)_{n\geq -1}$ is the sequence of Fibonacci\footnote{In the literature there are several ways \emph{Fibonacci numbers} are defined. Here we defined them in such a way, that they are the denominators of the convergents of the \emph{golden mean} $g$; see also~\eqref{goldenmean}.} numbers, given by
\begin{equation}\label{fibonacconumbers}
F_{-1}=0, F_0=1, F_n=F_{n-1}+F_{n-2},\,\, \text{for $n\in\N$}.
\end{equation}
For $x=g$, we have that the sequence $(\Theta_n)_{n\geq 0}$ converges in an alternating manner to
$$
\frac{g}{1+g^2} = \frac{1}{\sqrt{5}}.
$$
Therefore, it seems difficult to improve Borel's result~\eqref{Borel}; the approach from~\cite{[BK]} seems not applicable here. In order to improve the Borel-bound $\tfrac{1}{\sqrt{5}}$ we ``have to stay away'' from $(\tfrac{1}{\sqrt{5}},\tfrac{1}{\sqrt{5}})\in\Delta$, which is hard as this point is a fixed point of $F$. We will see that the ideas from~\cite{[BK]} can be adopted, except from some special cases. We will see that $x=g$ (or more generally, $t_n=g$ for some $n\in\N\cup\{ 0\}$) is a special case that we will treat separately in~Subsection~\ref{section5}. We will find a Han\u{c}l \& Nair-type result for this special case.\medskip\

In this section, as $t_n\neq g$, we assume that there are $\ell\in\N$ and $k\in\N$, $k\geq 2$, such that
$$
t_n = [0; \underbrace{1,1,\dots,1}_{\ell -\text{times}},k,\dots] =: [0;1^{\ell},k,\dots]
$$
(note that $\ell\geq 1$, as $a_{n+1}=1$). We need to further refine the cylinders $\Delta (1)$ and $\Delta (1,k)$ as follows. Let $\ell,k\in\N$, then
$$
\Delta (1^{\ell},k) = \{ x\in [0,1)\, \big{|}\, a_{n+1}=1,\dots,a_{n+\ell}=1, a_{n+\ell+1}=k\}.
$$
For $\ell\geq 2$ we obviously have that
$$
T(\Delta (1^{\ell},k)) = \Delta (1^{\ell -1},k),
$$
and $T(\Delta (1,k))=\Delta (k) = (\frac{1}{k+1},\frac{1}{k}]$ if $k\geq 2$, and $T(\Delta (1,1)) = \Delta (1) = (\frac{1}{2},1)$.
We furthermore define
$$
V(1^{\ell},k)=\Delta (1^{\ell},k)\times [0,1]\quad \text{and}\quad H(1^{\ell},k)=[0,1]\times \Delta (1^{\ell},k),
$$
and also
$$
V^*(1^{\ell},k) = \Psi(V(1^{\ell},k))\quad \text{and}\quad H^*(1^{\ell},k) = \Psi(H(1^{\ell},k)).
$$
Finally, define the sets $W_{\ell}$, $V_{\ell}'$ and $H_{\ell}'$ by
$$
W_{\ell}=\bigcup_{k\geq 2} \Delta (1^{\ell},k),\quad V_{\ell}'= \Psi(W_{\ell}\times [0,1]),\quad\text{and}\quad H_{\ell}'= \Psi([0,1]\times W_{\ell}).
$$
Note that $V_{\ell}'$ is the reflection of $H_{\ell}'$ in the diagonal $\beta=\alpha$, and vice versa.

\begin{remark}\label{remarkonWell}{\rm
Setting 
\begin{equation}\label{cell}
c_{\ell}=\frac{F_{\ell+1}}{F_{\ell+2}},\quad \text{for $\ell\geq -1$}
\end{equation}
where $(F_n)_{n\geq -1}$ are the Fibonacci numbers of~\eqref{fibonacconumbers}, one has
$$
\begin{array}{lclclclcl}
W_1 &=& \bigcup_{k\geq 2} \Delta (1,k) &=& [\tfrac{2}{3},1) &=& [\tfrac{F_2}{F_3},\tfrac{F_0}{F_1}) &=& [c_1,c_{-1})\\
W_2 &=& \bigcup_{k\geq 2} \Delta (1,1,k) &=& (\tfrac{1}{2},\tfrac{3}{5}] &=& (\tfrac{F_1}{F_2},\tfrac{F_3}{F_4}] &=&  (c_0,c_2]\\
W_3 &=& \bigcup_{k\geq 2} \Delta (1^3,k) &=& [\tfrac{5}{8},\tfrac{2}{3}) &=& [\tfrac{F_4}{F_5},\tfrac{F_2}{F_3}) &=& [c_3,c_1)\\
W_4 &=& \bigcup_{k\geq 2} \Delta (1^4,k) &=& (\tfrac{3}{5},\tfrac{8}{13}] &=& (\tfrac{F_3}{F_4},\tfrac{F_5}{F_6}] &=& (c_2,c_4]\\
\,\,\,\vdots && \qquad\,\, \vdots && \quad\vdots && \,\,\quad\vdots && \,\,\quad\vdots 
\end{array}
$$
From definition~\eqref{fibonacconumbers} of the Fibonacci numbers, from the fact that $T(W_{\ell+1})=W_{\ell}$, for $\ell\geq 1$, by induction and by~\eqref{cell}, we have if $\ell\in\N$ is \emph{odd}
$$
W_{\ell} = \bigcup_{k\geq 2} \Delta (1^{\ell},k) = \Big[ \frac{F_{\ell+1}}{F_{\ell+2}}, \frac{F_{\ell-1}}{F_{\ell}} \Big) = [c_{\ell},c_{\ell-2}),
$$
and if $\ell\in\N$ is \emph{even}
$$
W_{\ell} = \bigcup_{k\geq 2} \Delta (1^{\ell},k) = \Big( \frac{F_{\ell-1}}{F_{\ell}}, \frac{F_{\ell+1}}{F_{\ell+2}}\Big] = (c_{\ell-2},c_{\ell}].
$$
}\hfill $\triangle$
\end{remark}

\begin{figure}[h]
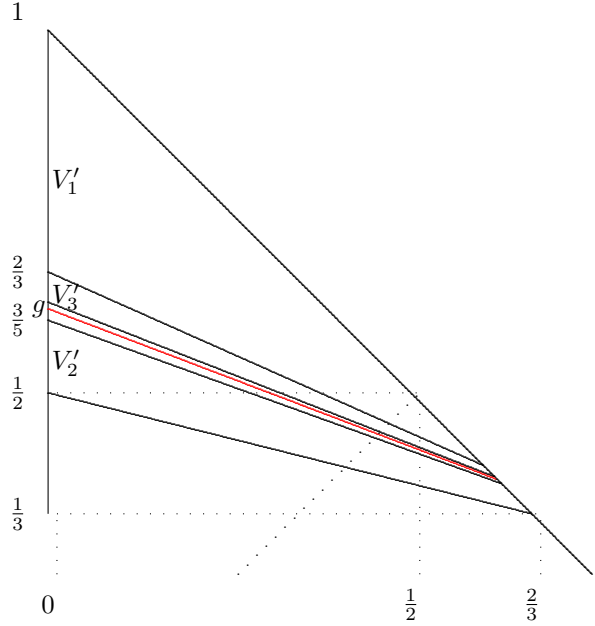

$$
\beginpicture
    \setcoordinatesystem units <0.8cm,0.8cm>
    \setplotarea x from 0 to 12, y from 0 to 9 %for y the scale is from 0 to 12 - 3, so 1->9, 1/2->3, etc.
    \putrule from 0 1 to 0 9
    \put {$1$} at -0.5 9.3
    \put {$0$} at 0 -0.5
    \put {$\tfrac{1}{2}$} at 6 -0.5
    \put {$\tfrac{2}{3}$} at 8 -0.5
    \put {$\tfrac{1}{2}$} at -0.5 3
    \put {$\tfrac{2}{3}$} at -0.5 5
    %\put {$\tfrac{3}{4}$} at -0.5 6
    \put {$\tfrac{1}{3}$} at -0.5 1
    \put {$g$} at -0.15 4.39
    \put {$\tfrac{3}{5}$} at -0.5 4.2
    \put {$V_1'$} at 0.3 6.5
    \put {$V_2'$} at 0.3 3.5
    \put {$V_3'$} at 0.3 4.59
  
\setlinear \plot
0 9 9 0
/
\setlinear \plot
0 3 8 1
/
\setlinear \plot
0 5 7.2 1.8
/
\setlinear \plot
0 4.2 7.5 1.5
/
\setlinear \plot
0 4.5 7.38 1.615
/
\setlinear {\color{red}\plot
0 4.39 7.4 1.58
/}

\setdots
\putrule from 0 1 to 8 1
\putrule from 0 3 to 6 3
\putrule from 6 0 to 6 3
\putrule from 0 0 to 0 1
\putrule from 8 0 to 8 1 
\setlinear \plot
3 0 6 3
/

\endpicture
$$ \caption[triangle]{$V_{\ell}'$ stacked ``on top of each other'' for $\ell =1,3$, and $\ell=2$. Note that the ``bottom line'' of $V_3'$ and the top-line of $V_2'$ are already very close to the line $\beta = -g^2\alpha+g$ (in {\color{red}red}).} \label{figurestackedVs}
\end{figure}

Note that for $\ell$ \emph{odd} the sets $V_{\ell}'$ are seamlessly ``stacked'' underneath each other (so $V_1'$ on top, followed by $V_3'$, et cetera), and ``converge'' to the line-segment $\beta =-g^2\alpha +g$; which is the image under $\Psi$ of the vertical line-segment $t=g$ in $\Omega$. For $\ell\geq 1$ \emph{even} the sets $V_{\ell}'$ are seamlessly ``stacked'' on top of each other, and ``converge'' to the line-segment $\beta =-g^2\alpha +g$. See also Remark~\ref{remarkonhorizontalandverticallines}. In Figure~\ref{figurestackedVs} we represent $V_1'$, $V_3'$, and $V_2'$. Obviously, it is not possible to depict the other $V_{\ell}'$, due to the ``fast'' convergence of the numbers $c_{\ell}$ to $g$, when $\ell\to\infty$.

For the sets $H_{\ell}'\cap \Psi (V_1)$ we see that these are stacked seamlessly next to each other; if $\ell$ is \emph{even}, these are stacked in increasing $\ell$, ``converging'' to the line (segment) $\beta = -G^2\alpha + G$, while for $\ell$ \emph{odd} these are stacked from right to left, again ``converging'' to $\beta = -G^2\alpha + G$ (this is also clear when one reflects the sets $V_{\ell}'$ in the diagonal $\beta = \alpha$).\smallskip\

As $\min\left\{ \Theta_{n-1},\Theta_n, \Theta_{n+1}\right\}$ in the sets $V_{\ell}'\cap H_h'$ ``behaves'' differently whether $\ell$ is odd or whether $\ell$ is even, and also whether $h$ is odd or even, when combining future $t_n$ and past $v_n$ at time $n$, this will lead to four cases which will be treated separately in the next section, Section~\ref{section:combining}.\medskip\

The point of intersection of the line $\beta = -c_{\ell}^2\alpha +c_{\ell}$ with $\alpha +\beta = 1$ is
$$
\left( \frac{1}{1+c_{\ell}}, \frac{c_{\ell}}{1+c_{\ell}}\right),
$$
and 
$$
\frac{1}{1+c_{\ell}}\uparrow g\quad \text{and}\quad \frac{c_{\ell}}{1+c_{\ell}}\downarrow g^2,\,\, \text{when $\ell$ \emph{odd} and $\ell\to\infty$},
$$
while the arrows are ``flipped'' when $\ell$ is \emph{even},
$$
\frac{1}{1+c_{\ell}}\downarrow g\quad \text{and}\quad \frac{c_{\ell}}{1+c_{\ell}}\uparrow g^2,\,\, \text{when $\ell$ \emph{even} and $\ell\to\infty$},
$$
Furthermore,
$$
f\left( \frac{1}{1+c_{\ell}}, \frac{c_{\ell}}{1+c_{\ell}}\right) = \frac{2(1-c_{\ell})}{1+c_{\ell}}\uparrow\downarrow 2g^3 = 0.4721\dots>\frac{1}{\sqrt{5}},\,\, \text{when $\ell$ \emph{odd/even} and $\ell\to\infty$}.
$$
The value of $f$ is easily found in the line segment $\alpha +\beta = 1$, by substituting $\beta = 1-\alpha$ in $f(\alpha,\beta)$,
\begin{eqnarray*}
f(\alpha ,\beta ) &=& \alpha +\sqrt{(1-4\alpha (1-\alpha))}-(1-\alpha)\\
&=& 2\alpha  - 1 +\sqrt{(2\alpha -1)^2}\\
&=& 2\alpha  - 1 + \left|2\alpha -1\right|\\
&=& \begin{cases}
2\alpha - 1 + 1 - 2\alpha = 0, & \text{if $0\leq\alpha\leq \tfrac{1}{2}$};\\
2\alpha -1 + 2\alpha -1 = 4\alpha -2,  & \text{if $\tfrac{1}{2}\leq\alpha\leq 1$}.
\end{cases}
\end{eqnarray*}

\noindent
On the line $\beta = -c_{\ell}^2\alpha +c_{\ell}$, with $\ell\geq 0$ (i.e., with $\tfrac{1}{2}\leq c_{\ell}<1$), we have
\begin{eqnarray*}
f(\alpha ,\beta ) &=& \alpha +\sqrt{(1-4\alpha (-c_{\ell}^2\alpha +c_{\ell}))}-(-c_{\ell}^2\alpha +c_{\ell})\\
&=& (1+c_{\ell}^2)\alpha  - c_{\ell} +\sqrt{(2c_{\ell}\alpha -1)^2}\\
&=& (1+c_{\ell}^2)\alpha  - c_{\ell} + (1-2c_{\ell}\alpha )\\
&=& (1-c_{\ell})^2\alpha  + 1 - c_{\ell}.
\end{eqnarray*}
So with $\alpha\in \left[0, \tfrac{1}{1+c_{\ell}}\right]$ and $(\alpha ,\beta )$ on the line $\beta = -c_{\ell}^2\alpha +c_{\ell}$ it follows that $f(\alpha ,\beta )$ is monotonically increasing
when $\ell\geq 0$, as $f(\alpha,\beta)$ for $(\alpha,\beta)$ on the line $\beta=-c_{\ell}^2\alpha +c_{\ell}$ is a linear function of $\alpha$ with positive derivative.\smallskip\

In Section~\ref{section:combining} the following will be very instrumental. There is a $\beta$-value of the point $(\alpha ,\beta )$ in the line $\beta = -c_{\ell}^2\alpha +c_{\ell}$ equal to $f(\alpha ,\beta )$. In order to find this point the following equation needs to be solved:
$$
-c_{\ell}^2\alpha  + c_{\ell} = (c_{\ell}-1)^2\alpha  +1 - c_{\ell},
$$
yielding
\begin{equation}\label{alpha*beta*}
\alpha^*=\frac{2c_{\ell}-1}{2c_{\ell}^2-2c_{\ell}+1}\downarrow \tfrac{1}{\sqrt{5}},\quad \text{and}\quad \beta^*=\frac{c_{\ell}(1-c_{\ell})}{2c_{\ell}^2-2c_{\ell}+1}\uparrow \frac{1}{\sqrt{5}},\quad \text{as $\ell$ is odd, $\ell\to\infty$},
\end{equation}
with the arrows ``flipped'' when $\ell$ is even. So for $\ell$ odd, with Lemma~\ref{valuesfon''horizontallines''} and Lemma~\ref{valuesfon''verticallines''} in mind (see below), on
$$
V_{\ell}'\setminus \left\{ (\alpha ,\beta )\in V_{\ell}'\, \big{|}\, \beta < \tfrac{c_{\ell}(1-c_{\ell})}{2c_{\ell}^2-2c_{\ell}+1} \right\}
$$
the function $f$ attains its maximum in $(\alpha^*,\beta^*)$. In Section~\ref{section:combining} similar points with first coordinate $\alpha^{\sharp}$ resp.\ $\alpha^{\flat}$ will also be introduced and put to work.\smallskip 

The following two lemmas will be used extensively in the next sections. We skip their proofs, as these proofs are essential the same as the above derivations of the behavior of $f$ on the lines $\beta = -c_{\ell}^2\alpha+c_{\ell}$, for $\ell\geq -1$.

\begin{lemma}\label{valuesfon''horizontallines''}
Let $c\in [\tfrac{1}{2},1)$ and $(\alpha ,\beta )\in\Psi (V_1)$, such that $\beta = -c^2\alpha+c$. Then for $\alpha\in \left[ 0,\tfrac{1}{1+c}\right]$ one has
\begin{equation}\label{fon''horizontal''line}
f(\alpha,\beta) = (c-1)^2\alpha + 1-c.
\end{equation}
\end{lemma} \smallskip\
So, on the line $\beta = -c^2\alpha +c$ the function $f$ is increasing monotonically. A similar result holds for line segments $\beta =-\tfrac{1}{d^2}\alpha +\tfrac{1}{d}$ in $\Psi (V_1)$, for $d\in (0,1)$.
\begin{lemma}\label{valuesfon''verticallines''}
Let $d\in (0,1]$ and $(\alpha ,\beta )\in\Psi (V_1)$, such that $\beta = -\tfrac{1}{d^2}\alpha +\tfrac{1}{d}$. Then for $\alpha\in \left[\tfrac{d}{d+1},\tfrac{2d}{2+d}\right]$ one has
\begin{equation}\label{fon''vertical''line}
f(\alpha ,\beta ) = \left( \frac{1}{d}+1\right)^2\alpha  -  \left( \frac{1}{d}+1\right).
\end{equation}
\end{lemma}

\begin{remark}\label{remark2}{\rm
In $\Delta$, the line-segment $\beta = -\tfrac{1}{d^2}\alpha +\tfrac{1}{d}$ has endpoints $(d,0)$ and $\left( \tfrac{d}{d+1},\tfrac{1}{d+1}\right)$. Since $a_{n+1}=1$, only the part of $y=-\tfrac{1}{d^2}x+\tfrac{1}{d}$ in $\Psi (V_1)$, which has endpoints $\left( \tfrac{d}{d+1},\tfrac{1}{d+1}\right)$ and $\left(\tfrac{2d}{2+d},\tfrac{1}{2+d}\right)$, the last point being the point of intersection of $\beta =-\tfrac{1}{d^2}\alpha +\tfrac{1}{d}$ and $\beta = -\tfrac{1}{4}\alpha +\tfrac{1}{2}$ is being considered (the lower boundary line of $\Psi (V_1)$). Note that if $d=1$ the line segment $\alpha+\beta=1$, with $\tfrac{1}{2}\leq \alpha\leq \tfrac{2}{3}$. For these $\alpha$ and $\beta = 1-\alpha$ it holds that $f(\alpha,\beta)=4\alpha -2$.\smallskip\
\newline
For $\alpha = \tfrac{d}{d+1}$ and $\beta = \tfrac{1}{d+1}$,~\eqref{fon''vertical''line} gives
$$
f(\alpha ,\beta ) = \left( \frac{1}{d}+1\right)^2 \cdot \frac{d}{d+1} -  \left( \frac{1}{d}+1\right) = 0,
$$
as expected, as $f(\alpha ,\beta )=0$ for $\alpha +\beta =1$ and $0\leq \alpha \leq \tfrac{1}{2}$. The maximum value of $f$ on the line segment $\beta =-\tfrac{1}{d^2}\alpha +\tfrac{1}{d}$ in $\Psi (V_1)$ is
$$
\left( \frac{1}{d}+1\right)^2 \cdot \frac{2d}{d+2} -  \left( \frac{1}{d}+1\right) = \frac{d+1}{d+2}.
$$
}\hfill $\triangle$
\end{remark}

\section{Combining past and future}\label{section:combining}
As we mentioned earlier, throughout Sections~\ref{OddOdd}--\ref{EvenEven} we assume $a_{n+1}=1=a_n$, as otherwise Borel's result~\eqref{Borel} is already much improved. In these sections, we  assume that $t_n\neq g$, so there exist an $\ell\geq 1$ and a $k\geq 2$ such that $t_n=[0;1^{\ell},k,\dots]$, and $v_n=[0;1^{h},a_{n-h}\neq 1,\dots, a_1]$, for some $h\in\{1,\dots,n-1\}$, or $v_n=[0;1^n]$. 
\noindent
Now in order to improve Borel's result for all $\ell\geq 1$ and $h\geq 1$, a constant $B_{\ell,h}<\frac{1}{\sqrt{5}}$ must be determined such that for all $(\alpha,\beta)\in V_{\ell}'\cap H_h' = \Psi (W_{\ell}\times W_h)$,
$$
\min \left\{\alpha, \beta, f(\alpha,\beta)\right\} \leq B_{\ell,h}.
$$
Obviously, we have four cases to consider: $(\ell,h)$ is either (\emph{odd}, \emph{odd}), (\emph{odd}, \emph{even}), (\emph{even}, \emph{odd}), or (\emph{even}, \emph{even}). We could consider a figure based on Figure~\ref{figurestackedVs}, where we intersect the $V_{\ell}'$ with their reflections on the diagonal $\beta = \alpha$. However, since the sets $V_{\ell}'$ are hardly visible in Figure~\ref{figurestackedVs} for $\ell \geq 3$, we refrain from doing so. Note that each time $\ell = h$, we have $\Psi (W_{\ell}\times W_h)$ is symmetric on the diagonal $\beta = \alpha$, and each time $\ell\neq h$, the reflection of $\Psi (W_{\ell}\times W_h)$ on the diagonal $\beta = \alpha$ yields $\Psi (W_h\times W_{\ell})$, and vice versa.\smallskip\

\noindent
Finally, in Section~\ref{section5} Borel's result~\eqref{Borel} will be sharpened for the special case $t_n=g$, where $n\in\N\cup\{ 0\}$.

\subsection{$(\ell,h)$ is \rm{(\emph{odd}, \emph{odd})}}\label{OddOdd} In this case we have, see also Remark~\ref{remarkonWell}, that 
$$
W_{\ell}\times W_h = [c_{\ell},c_{\ell-2}) \times [c_h,c_{h-2}).
$$
We first discuss the case $\ell=h$ \emph{odd}, then the two cases that $\ell\neq h$ \emph{odd}.\medskip\

($i$) If $\ell = h=1$ \emph{odd}, we see that $\Psi (W_{\ell}\times W_h)$ is a triangle, symmetric on the diagonal $\beta = \alpha$, with vertices $(\tfrac{6}{13},\tfrac{6}{13})$, $(\tfrac{2}{5},\tfrac{3}{5})$, and $(\tfrac{3}{5},\tfrac{2}{5})$. If $\ell = h\geq 3$ \emph{odd}, $\Psi (W_{\ell}\times W_h)$ is a quadrilateral, symmetric on the diagonal $\beta = \alpha$, with vertices
\begin{equation}\label{ell=hgeq3odd}
\left( \frac{c_{\ell}}{1+c_{\ell}^2},\frac{c_{\ell}}{1+c_{\ell}^2}  \right),\,\, \left( \frac{c_{\ell}}{1+c_{\ell -2}c_{\ell}},\frac{c_{\ell -2}}{1+c_{\ell -2}c_{\ell}}\right),\,\, \left(\frac{c_{\ell -2}}{1+c_{\ell -2}^2}, \frac{c_{\ell -2}}{1+c_{\ell -2}^2}\right),\,\, \left( \frac{c_{\ell -2}}{1+c_{\ell -2}c_{\ell}}, \frac{c_{\ell}}{1+c_{\ell -2}c_{\ell}}\right)
\end{equation}
which are the images of the vertices $A=(c_{\ell},c_{\ell})$, $B=(c_{\ell-2},c_{\ell})$, $C=(c_{\ell-2},c_{\ell-2})$, and $D=(c_{\ell},c_{\ell-2})$ resp.\ of the rectangle $W_{\ell}\times W_h$. See also Figure~\ref{ch4.1figure} below.\medskip\

\begin{figure}[h]
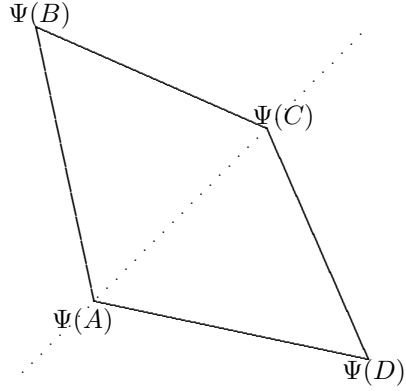

$$
\beginpicture
    \setcoordinatesystem units <15cm,15cm>
    \setplotarea x from 4.4 to 4.75, y from 4.4 to 4.75 
    \put {$\Psi (A)$} at 4.453 4.445
    \put {$\Psi (B)$} at 4.415 4.715
    \put {$\Psi (C)$} at 4.63 4.625
    \put {$\Psi (D)$} at 4.71 4.4
  
\setlinear \plot
4.463 4.463 4.4117 4.705 
/
\setlinear \plot
4.463 4.463 4.705 4.4117
/
\setlinear \plot
4.4117 4.705 4.61538 4.61538
/
\setlinear \plot
4.705 4.4117 4.61538 4.61538
/

\setdots
\setlinear \plot
4.4 4.4 4.7 4.7
/

\endpicture
$$ \caption[diagonal]{$\Psi (W_{\ell}\times W_{\ell})$ for $\ell=3$. See also~\eqref{ell=hgeq3odd} for the various points $\Psi(A)$, $\Psi(B)$, $\Psi(C)$, and $\Psi(D)$. The dotted line is the diagonal $\beta=\alpha$ in $\Delta$.}\label{ch4.1figure}
\end{figure}

($i_a$) The case $\ell=1=h$. Note that for $(\alpha,\beta)\in\Psi (W_1\times W_1)$ we have both $\alpha\geq \tfrac{2}{5}$ and $\beta\geq \tfrac{2}{5}$, and
$$
0\leq f(\alpha,\beta) \leq f(\tfrac{3}{5},\tfrac{2}{5}) = \tfrac{2}{5}.
$$
We see that for $(t_n,v_n)\in W_1\times W_1$,
$$
\min \{ \Theta_{n-1},\Theta_n,\Theta_{n+1}\} \leq B_{1,1} = \tfrac{2}{5}.
$$

($i_b$) The case $\ell=h\geq 3$ is \emph{odd}. From~\eqref{ell=hgeq3odd} we see that for $(\alpha,\beta)\in \Psi (W_{\ell}\times W_{\ell})$ both $\alpha \geq \frac{c_{\ell}}{1+c_{\ell -2}c_{\ell}}$ and  $\beta \geq \frac{c_{\ell}}{1+c_{\ell -2}c_{\ell}}$. From the definition of Fibonacci numbers~\eqref{fibonacconumbers}, definition~\eqref{cell} of the numbers $c_{\ell}$, and from the well-known recurrence relations for the sequence $(q_n)$ we find,
\begin{equation}\label{cell2}
c_{\ell} = \frac{1}{1+\displaystyle \frac{1}{1+ c_{\ell -2}}} = \frac{1+c_{\ell -2}}{2+c_{\ell -2}},\quad \text{for $\ell\in\N$}.
\end{equation}
Setting $h(c)=\tfrac{1+c}{c^2+2c+2}$ we have $h'(c)=\tfrac{-c(c+2)}{(c^2+2c+2)^2}<0$, and since for $\ell$ odd we have $c_{\ell}\downarrow g$ as $\ell\to\infty$, \eqref{cell2} yields
$$
\frac{c_{\ell}}{1+c_{\ell -2}c_{\ell}}=h(c_{\ell -2})\uparrow \frac{g}{1+g^2}=\frac{1}{\sqrt{5}},\quad \text{for $\ell$ odd, $\ell\to\infty$}.
$$
For points $(\alpha,\beta)\in \Psi (W_{\ell}\times W_{\ell})$,
$$
\frac{c_{\ell}-c_{\ell -2}+1-c_{\ell -2}c_{\ell}}{1+c_{\ell -2}c_{\ell}} \leq f(\alpha,\beta)\leq \frac{c_{\ell-2}-c_{\ell}+1-c_{\ell -2}c_{\ell}}{1+c_{\ell -2}c_{\ell}}.
$$
But then it follows from~\eqref{cell2}, that
\begin{equation}\label{cell3}
c_{\ell-2}-c_{\ell}+1-c_{\ell -2}c_{\ell} = c_{\ell-2} -\frac{1+c_{\ell-2}}{2+c_{\ell-2}} +1 -c_{\ell-2}\cdot \frac{1+c_{\ell-2}}{2+c_{\ell-2}} = \frac{1+c_{\ell-2}}{2+c_{\ell-2}}=c_{\ell},
\end{equation}
and we see that in this case (i.e., $(\alpha,\beta)\in \Psi (W_{\ell}\times W_{\ell})$) the minimal value for both $\alpha$ and $\beta$ is also the maximum possible value for $f(\alpha,\beta)$. Thus, from ($i_a$) and ($i_b$) the following (partial) result follows.
\begin{proposition}\label{ell=h odd}
Let $\ell=h\geq 1$ be odd. Then for $(t_n,v_n)\in W_{\ell}\times W_{\ell}$,
$$
\min \{ \Theta_{n-1},\Theta_n,\Theta_{n+1}\} \leq B_{\ell,\ell} = \frac{c_{\ell}}{1+c_{\ell -2}c_{\ell}} < \frac{1}{\sqrt{5}}.
$$
Equivalently, if $(\alpha,\beta)\in W_{\ell}\times W_{\ell} $, where $\ell\geq 1$ is odd, then 
$$
\min\left\{ \alpha,\beta,f(\alpha,\beta)\right\} \le B_{\ell,\ell} = \frac{c_{\ell}}{1+c_{\ell -2}c_{\ell}} < \frac{1}{\sqrt{5}}.
$$
These constants are best possible. I.e., they cannot be replaced by smaller ones.
\end{proposition}\medskip\

($ii$) Now assume that $\ell > h\geq 1$, where $\ell$ and $h$ are \emph{odd}. Note that $W_{\ell}\times W_h$ is a rectangle in $\Omega$ ``above'' the diagonal $v=t$ (as $g<c_{\ell}<c_h$), with vertices
$$
A:\ (c_{\ell},c_h),\,\, B:\, (c_{\ell -2},c_h),\,\, C:\, (c_{\ell -2},c_{h-2}),\,\, \text{ and } D:\, (c_{\ell},c_{h-2}),
$$
which are the bottom-left, bottom-right, top-right resp.\ top-left vertices of $W_{\ell}\times W_h$. Under $\Psi$, cf.~\eqref{PsiMap}, we see that $W_{\ell}\times W_h$ is mapped to a quadrilateral ``under'' the diagonal $\beta = \alpha$ in $\Delta$, with vertices
$$
\Psi (A)=\left( \frac{c_h}{1+c_{\ell}c_h}, \frac{c_{\ell}}{1+c_{\ell}c_h}\right),\quad \Psi (D) = \left( \frac{c_{h-2}}{1+c_{\ell}c_{h-2}}, \frac{c_{\ell}}{1+c_{\ell}c_{h-2}}\right),
$$
and
$$
\Psi (C) = \left( \frac{c_{h-2}}{1+c_{\ell -2}c_{h-2}}, \frac{c_{\ell-2}}{1+c_{\ell -2}c_{h-2}}\right),\quad \Psi (B) = \left( \frac{c_h}{1+c_{\ell-2}c_h}, \frac{c_{\ell -2}}{1+c_{\ell-2}c_h} \right),
$$
where $\Psi (A)$, $\Psi (D)$, $\Psi (C)$ resp.\ $\Psi (B)$ are the lower-left, lower-right, top-right resp.\ top-left vertices of $\Psi (W_{\ell}\times W_h)$. Since $\Psi (W_{\ell}\times W_h)$ is ``under'' the diagonal $\beta = \alpha$, we see that for $(\alpha,\beta)\in\Psi (W_{\ell},W_h)$ we have $\beta \leq \alpha$ and $\alpha \geq \frac{c_h}{1+c_{\ell-2}c_h}>\tfrac{1}{\sqrt{5}}$. Furthermore, from~\eqref{derivatesarenegative}, Lemma~\ref{valuesfon''horizontallines''}, and Lemma~\ref{valuesfon''verticallines''} it follows that in $\Psi (W_{\ell}\times W_h)$ the function $f$ takes its maximum in $\Psi (D)$, the maximum being
$$
f(\Psi (D)) = \frac{(1-c_{\ell})(1+ c_{h-2})}{1+c_{\ell}c_{h-2}}.
$$
\begin{example}\label{ex:1}{\rm 
If $\ell = 3$, $h=1$, we have (where we round off some of the fractions)
$$
\Psi (A) = \left( \tfrac{16}{34},\tfrac{15}{34}\right) = (0.47058,0.44117),\quad \Psi (D) = \left( \tfrac{8}{13},\tfrac{5}{13} \right) = (0.61538,0.38461),
$$
and
$$
\Psi (C) = \left( \tfrac{3}{5},\tfrac{2}{5}\right) = (0.6,\, 0.4),\quad \Psi (B) = \left( \tfrac{6}{13},\tfrac{6}{13} \right) = (0.46153,0.46153).
$$
Note that
$$
f(\Psi (D)) = f(\tfrac{8}{13},\tfrac{5}{13}) = \tfrac{6}{13} > \tfrac{5}{13},
$$
where $\tfrac{5}{13}$ is the second coordinate of $\Psi (D)$. We saw in~Lemma~\ref{valuesfon''horizontallines''} that for points $(\alpha ,\beta )$ in the bottom line $\beta = -\tfrac{25}{64}\alpha +\tfrac{5}{8}$of $V_3'$, with $0\leq \alpha\leq \tfrac{8}{13}$, we have
$$
f(\alpha ,\beta ) = \tfrac{9}{64}\alpha + \tfrac{3}{8}.
$$
So in the bottom line $\beta = -\tfrac{25}{64}\alpha +\tfrac{5}{8}$ there is a point $(\alpha^*,\beta^*)$ where the value of $\beta$ is equal to $f(\alpha ,\beta )$; see~\eqref{alpha*beta*}.  This intersection point is here $(\tfrac{8}{17}, \tfrac{15}{34})$, which is $\Psi (A)$. Since $f\left( \tfrac{8}{17}, \tfrac{15}{34} \right) = \tfrac{15}{34}$, we see that if $(t_n,v_n)\in W_3\times W_1$,
$$
\min \{ \Theta_{n-1},\Theta_n,\theta_{n+1}\} \leq B_{3,1} = \tfrac{15}{34}.
$$
}\hfill $\triangle$
\end{example}

Actually, what we find in Example~\ref{ex:1} is not a coincidence! We will use the following lemma.
\begin{lemma}\label{lemmaOnh=ell-20dd}
Let $\ell, h$ be odd, such that $h=\ell -2\geq 1$. Then
\begin{equation}\label{h=ell-20dd}
\frac{c_h}{1+c_{\ell}c_h} = \frac{2c_{\ell}-1}{2c_{\ell}^2-2c_{\ell}+1}. 
\end{equation}
\end{lemma}

\begin{proof}
Setting $c=c_h=c_{\ell -2}$, we have from~\eqref{cell2} that $c_{\ell}=\frac{1+c}{2+c}$. But then we have
\begin{equation}\label{eq1:lemmaOnh=ell-20dd}
\frac{c_h}{1+c_{\ell}c_h} = \frac{c}{1+c\cdot \frac{1+c}{2+c}} = \frac{c(2+c)}{c^2 + 2c + 2}.
\end{equation}
Furthermore,
\begin{eqnarray*}
\frac{2c_{\ell}-1}{2c_{\ell}^2 - 2c_{\ell} +1} &=& \frac{\frac{2+2c}{2+c}-1}{2\left(  \frac{1+c}{2+c}\right)^2 -  \frac{2(1+c)}{2+c} + \frac{(2+c)^2}{(2+c)^2}}\\
&=& \frac{c(c+2)}{2(1+c)^2 -2(1+c)(2+c) + (2+c)^2}\\
&=& \frac{c(2+c)}{c^2 + 2c + 2}.
\end{eqnarray*}
From this and~\eqref{eq1:lemmaOnh=ell-20dd} the result~\eqref{h=ell-20dd} follows.
\end{proof}

Thus we see that $\alpha^*$ from~\eqref{alpha*beta*} is equal to the first coordinate of $\Psi (A)$, and since $(\alpha^*,\beta^*)$ and $\Psi (A)$ are on the same line $\beta = -c_{\ell}^2\alpha+c_{\ell}$, we must have that both points are equal and that, in particular, the second coordinate of both points must be equal. Therefore, an immediate consequence of this lemma is that if $(t_n,v_n)\in W_{\ell}\times W_h$, with $\ell , h$ \emph{odd} and $\ell > h\geq 1$, we find that $\Theta_n> \frac{c_{\ell}(1-c_{\ell})}{2c_{\ell}^2-2c_{\ell}+1}$, but then $\Theta_{n+1} = f(\Theta_{n-1},\Theta_{n+1})\leq \frac{c_{\ell}(1-c_{\ell})}{2c_{\ell}^2-2c_{\ell}+1}$, or $\Theta_n\leq \frac{c_{\ell}(1-c_{\ell})}{2c_{\ell}^2-2c_{\ell}+1}$. Note that in this case, always $\Theta_{n-1}>\tfrac{1}{\sqrt{5}}$. Thus, we find for $n\in\N$ and $(t_n,v_n)\in W_{\ell}\times W_h$, with $\ell, h$ \emph{odd} and $\ell > h\geq 1$,
$$
\min \{ \Theta_{n-1},\Theta_n,\Theta_{n+1}\}\leq B_{\ell,h} = \frac{c_{\ell}(1-c_{\ell})}{2c_{\ell}^2-2c_{\ell}+1} = \beta^* <\frac{1}{\sqrt{5}}.
$$
Note that $\Psi (A)$ is the lower left vertex of $\Psi (W_{\ell}\times W_{\ell -2})$, and the second coordinate of $\Psi (A)$ is given by $\frac{c_{\ell}}{1+c_{\ell}c_{\ell -2}}$. Furthermore, if $\ell\geq 5$ \emph{odd}, the upper left vertex of $\Psi (W_{\ell}\times W_{\ell -4})$ is given by $( \frac{c_{\ell-4}}{1+c_{\ell-2}c_{\ell-4}}, \frac{c_{\ell-2}}{1+c_{\ell-2}c_{\ell-4}})$. We have the following lemma.

\begin{lemma}\label{ComparisonEndpoints}
For $\ell\geq 3$ odd, we have the following,
$$
\frac{c_{\ell}}{1+c_{\ell}c_{\ell -2}} > \frac{c_{\ell-2}}{1+c_{\ell-2}c_{\ell-4}}.
$$
\end{lemma}

\begin{proof}
Replacing $\ell$ by $\ell-2$ in~\eqref{cell2}, we find
$$
c_{\ell-2} = \frac{1+c_{\ell -4}}{2+c_{\ell -4}};
$$
in view of this, we define a function $s$ (on an appropriate domain, e.g.\ $[0,1]$) as
$$
s(x) = \frac{\frac{x+1}{x+2}}{1+x\cdot \frac{x+1}{x+2}} = \frac{x+1}{x^2+2x+2}.
$$
Obviously, $s'(x)<0$ for $x>0$, and since $c_{\ell}<c_{\ell-2}$ (recall that $\ell$ is odd!), we see that
$$
\frac{c_{\ell}}{1+c_{\ell}c_{\ell -2}}=s(c_{\ell-2}) > s(c_{\ell-4}) = \frac{c_{\ell-2}}{1+c_{\ell-2}c_{\ell-4}} ,
$$
which is what we set out to prove.
\end{proof}

An immediate consequence of Lemma~\ref{ComparisonEndpoints} is the following (partial) result. Recall that $\frac{c_{\ell -2}}{1+c_{\ell-2}c_h}$ is the second coordinate of the top left vertex of $\Psi (W_{\ell}\times W_h)$.
\begin{proposition}\label{ell>hgeq1odd}\label{ell and h both odd ell lager than h} Let $\ell ,h$ be odd positive integers, where $\ell >h\geq 1$. Let $(t_n,v_n)\in W_{\ell}\times W_h$, i.e., $(\Theta_{n-1},\Theta_n)\in \Psi (W_{\ell}\times W_h)$. Then
\begin{enumerate}
\item[($i$)]
If $h=\ell-2$,
$$
\min \{ \Theta_{n-1},\Theta_n,\Theta_{n+1}\}\leq B_{\ell,\ell-2} = \frac{c_{\ell}(1-c_{\ell})}{2c_{\ell}^2-2c_{\ell}+1} = \beta^* <\frac{1}{\sqrt{5}}.
$$
\item[($ii$)]
If $1\leq h < \ell-2$, 
$$
\min \{ \Theta_{n-1},\Theta_n,\Theta_{n+1}\}\leq B_{\ell,h} = \frac{c_{\ell -2}}{1+c_{\ell-2}c_h} < \frac{c_{\ell}(1-c_{\ell})}{2c_{\ell}^2-2c_{\ell}+1} <\frac{1}{\sqrt{5}}.
$$
\end{enumerate}
These constants are best possible. I.e., they cannot be replaced by smaller ones.
\end{proposition}

\begin{proof}
Recall that the lower right vertex of $\Psi (W_{\ell}\times W_{\ell})$ has as second coordinate $\beta^*=\frac{c_{\ell}}{1+c_{\ell-2}c_{\ell}}$, and that the lower right vertex of $\Psi (W_{\ell}\times W_{\ell})$ is also the lower left vertex of $\Psi (W_{\ell}\times W_{\ell-2})$. So, on $\Psi (W_{\ell}\times W_{\ell-2})$ we cannot improve our earlier result, Proposition~\ref{ell=h odd} (and this is part (1) of the present Proposition). However, note that the upper left vertex of $\Psi (W_{\ell}\times W_{\ell -2})$ (so $h=\ell -2$) has as second coordinate $\tfrac{c_{\ell-2}}{1+c_{\ell-2}c_{\ell-4}}$, and due to Lemma~\ref{ComparisonEndpoints} we know that this value is \emph{smaller} than $\beta^*$. Since for $1\leq h\leq \ell-2$ the value of the second coordinate of the upper left vertex of $\Psi (W_{\ell}\times W_h)$ decreases as $h$ odd decreases from $\ell-2$ to $1$ (these upper left vertices are all on the same monotonically decreasing line $\beta = -c_{\ell-2}^2\alpha+c_{\ell-2}$), we find that statement ($ii$) holds.
\end{proof}\smallskip\

\begin{example}\label{NotUsingf}{\rm
Let $\ell = 7$ and $h=3$. As we saw before, the upper left vertex of $\Psi (V_7\times W_3)$ is\footnote{Again, we rounded off some of the fractions.} $\left( \tfrac{105}{233}, \tfrac{104}{233}\right) = (0.450643, 0.44635193)$, while the lower left vertex of $\Psi (V_7\times W_3)$ is $\left( \tfrac{55}{122}, \tfrac{136}{305} \right)=(0.4508196, 0.445901639)$. One easily shows that $\beta^*=\tfrac{714}{1597}=0.44708829$, and since
$$
\tfrac{104}{233} = 0.44635193 < 0.44708829 = \tfrac{714}{1597},
$$
using the second coordinate of the upper left vertex of $\Psi (V_7\times W_3)$ given an improvement over $\beta^*$. Due to Proposition~\ref{ell>hgeq1odd} we know that using $f$ on the boundary of $\Psi (V_7\times W_3)$ does not give any improvement. Indeed, as $f\left( \tfrac{105}{233}, \tfrac{104}{233}\right) = \tfrac{104}{233}$, which is the second coordinate of this upper left hand vertex, and on $\Psi (V_7\times W_3)$ this is the lowest value $f$ attains due to Lemmas~\ref{valuesfon''horizontallines''}  and~\ref{valuesfon''verticallines''}. In fact, on $\Psi (V_7\times W_1)$ the situation is even worse! The top left vertex of this set is the top right vertex of $\Psi (V_7\times W_3)$, and is given by 
$$
\Psi \left( \tfrac{34}{55}, \tfrac{2}{3}\right) = \left( \tfrac{110}{233},\tfrac{102}{233}\right) = (0.472,0.43776824).
$$
However, $f\left( \tfrac{110}{233},\tfrac{102}{233}\right) = \tfrac{105}{233} > \tfrac{102}{233}$, so on $\Psi (V_7\times W_1)$ the lowest value $f$ attains on this set is \emph{greater} than the highest possible value of $\beta$ when $(\alpha,\beta)\in\Psi (V_7\times W_1)$. That this would be the case is indicated in the proof of Proposition~\ref{ell>hgeq1odd}.}\hfill $\triangle$
\end{example}

\begin{figure}[h]
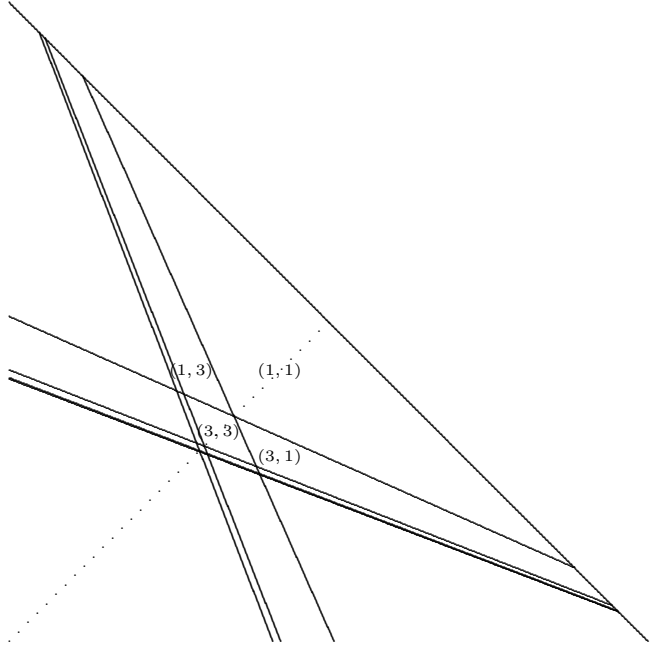

$$
\beginpicture
    \setcoordinatesystem units <3.25cm,3.25cm>
    \setplotarea x from 2.4 to 10, y from 3.7 to 6.3
\put {\tiny{$(1,1)$}} at 4.8 4.8
\put {\tiny{$(1,3)$}} at 4.44 4.8
\put {\tiny{$(3,1)$}} at 4.8 4.45
\put {\tiny{$(3,3)$}} at 4.55 4.55
\setlinear \plot
6.3 3.7 3.7 6.3
/
\setlinear \plot
6 4 3.7 5.022222
/

\setlinear \plot
4 6 5.022222 3.7
/
\setlinear \plot
6.1538 3.84615 3.7 4.8046875
/
\setlinear \plot
3.84615 6.1538 4.8046875 3.7
/
\setlinear \plot
6.17647 3.82352941 3.7 4.77256235
/
\setlinear \plot
3.82352941 6.17647 4.77256235 3.7
/
\setlinear \plot
6.17977528 3.82022471 3.7 4.76786776
/
%\setlinear \plot
%3.82022471 6.17977528 4.76786776 3.7
%/

\setdots
\setlinear \plot
3.7 3.7 5 5
/
\endpicture
$$ \caption[triangle]{The sets $\Psi (V_{\ell}\times H_h)$, for $\ell, h\in \{ 1,3,5\}$. For $\ell, h\in \{ 1,3\}$ we indicated these sets by $(\ell,h)$. Still visible (but not indicated) are the sets $(5,5)$, $(5,3)$, $(5,1)$, $(3,5)$, and $(1,5)$. In case $\ell$ or $h$ are equal to 7 (or larger) there is ``no space''. These sets are ``thinner'' than pixel width and become a ``fat line'' in this figure.}\label{ImahesOfPsi(VtimesH}
\end{figure}
\smallskip\

\noindent
($iii$) Now assume that $h>\ell \geq 1$, where $\ell$ and $h$ are odd. As in case $(ii$), note that $W_{\ell}\times W_h$ is a rectangle in $\Omega$, but now  ``under'' the diagonal $v=t$ (as $g<c_h<c_{\ell}$), with vertices $A:\ (c_{\ell},c_h)$, $B:\, (c_{\ell -2},c_h)$, $C:\, (c_{\ell -2},c_{h-2})$, and $D:\, (c_{\ell},c_{h-2})$, which are the bottom-left, bottom-right, upper-right resp.\ upper-left vertices of $W_{\ell}\times W_h$. Under $\Psi$, cf.~\eqref{PsiMap}, we see that $W_{\ell}\times W_h$ is mapped to a quadrilateral ``above'' the diagonal $\beta = \alpha$ in $\Delta$, with vertices similar to the ones before:
$\Psi (A)$, $\Psi (D)$, $\Psi (C)$ resp.\ $\Psi (B)$ are again the lower-left, lower-right, upper-right resp.\ upper-left vertices of $\Psi (W_{\ell}\times W_h)$.

Obviously, we have
$$
\bigcup_{h=\ell+2}^{\infty} \Psi ( [c_{\ell},c_{\ell-2}]\times [c_h,c_{h-2}]) = \Psi ( [c_{\ell},c_{\ell-2}]\times [g,c_{\ell}]),
$$
where $\Psi ( [c_{\ell},c_{\ell-2}]\times [g,c_{\ell}])$ is a quadrilateral with vertices 
$$
\Psi (A^*) = \Psi (c_{\ell},g) = \left( \tfrac{g}{1+c_{\ell}g}, \tfrac{c_{\ell}}{1+c_{\ell}g}\right),\quad \Psi (D^*) = \Psi (c_{\ell},c_{\ell})=\left( \tfrac{c_{\ell}}{1+c_{\ell}^2},\tfrac{c_{\ell}}{1+c_{\ell}^2}\right).
$$
which are the lower left, resp.\ lower right vertices of $\Psi ( [c_{\ell},c_{\ell-2}]\times [g,c_{\ell}])$, and 
$$
\Psi (B^*) = \Psi (c_{\ell-2},g) = \left( \tfrac{g}{1+c_{\ell-2}g}, \tfrac{c_{\ell-2}}{1+c_{\ell-2}g}\right), \quad \Psi (C^*) = \Psi (c_{\ell-2},c_{\ell}) = \left( \tfrac{c_{\ell}}{1+c_{\ell-2}c_{\ell}}, \tfrac{c_{\ell-2}}{1+c_{\ell-2}c_{\ell}}  \right),
$$ 
which are the top left, resp.\ top right vertices of $\Psi ( [c_{\ell},c_{\ell-2}]\times [g,c_{\ell}])$. This quadrilateral is bounded by straight lines $\beta = -c_{\ell}^2\alpha+c_{\ell}$ (bottom), $\beta = -c_{\ell-2}^2\alpha+c_{\ell-2}$ (top),  $\beta = -G^2\alpha +G$ (LHS), and $\beta = -\tfrac{1}{c_{\ell}^2}\alpha +\tfrac{1}{c_{\ell}}$ (RHS).

From~\eqref{derivatesarenegative} and lemmas~\ref{valuesfon''horizontallines''} and~\ref{valuesfon''verticallines''} it is clear that on $\Psi ( [c_{\ell},c_{\ell-2}]\times [g,c_{\ell}])$) the function $f$ is maximal in the lower right-hand vertex $\left( \tfrac{c_{\ell}}{1+c_{\ell}^2},\tfrac{c_{\ell}}{1+c_{\ell}^2}\right)$, and that
$$
f \left( \tfrac{c_{\ell}}{1+c_{\ell}^2},\tfrac{c_{\ell}}{1+c_{\ell}^2}\right) = \frac{1-c_{\ell}^2}{1+c_{\ell}^2}.
$$
Note that for $\ell$, $h$ odd, $h>\ell\geq 1$, this part would be proved and ``don'' if on each of the quadrilaterals $\Psi ([c_{\ell},c_{\ell-2}]\times [c_h,c_{h-2}])$, the first coordinate $\tfrac{g}{1+c_{\ell-2}g}$ of $\Psi (c_{\ell-2},g)$ were larger than $\frac{1-c_{\ell}^2}{1+c_{\ell}^2}$. However, this is not the case. For example, for $\ell=1$, $h=3$ one has
$$
\frac{1-c_{\ell}^2}{1+c_{\ell}^2} = \frac{1-\tfrac{4}{9}}{1+\tfrac{4}{9}} = \tfrac{5}{13} = 0.384615384\cdots\, ,
$$
while
$$
\tfrac{g}{1+c_{\ell-2}g} = \tfrac{g}{1+g} = g^2 = 0.381966011\cdots < \tfrac{5}{13}.
$$
The first coordinate of the top-left-hand vertex of $\Psi ([c_{\ell},c_{\ell-2}]\times [c_h,c_{h-2}])$ in this case (i.e., $\ell =1$, $h=\ell+2=3$) is
$$
\frac{c_h}{1+c_hc_{\ell-2}} = \frac{\tfrac{5}{8}}{1+\tfrac{5}{8}} = \frac{1-c_{\ell}^2}{1+c_{\ell}^2},
$$
so for $\ell = 1 = h-2$, for all $(\alpha,\beta)\in\Psi ([c_{\ell},c_{\ell-2}]\times [c_h,c_{h-2}])$ we have $\alpha\geq \tfrac{5}{13},\,\, \beta \gg \tfrac{5}{13},\,\, \text{and }\, f(\alpha,\beta) \leq \tfrac{5}{13}$, and therefore
$$
\min \{ \alpha,\beta,f(\alpha,\beta)\} \leq B_{1,3} = \tfrac{5}{13} < \tfrac{1}{\sqrt{5}}.
$$
In order to find this result for general $\ell$, $h$ \emph{odd}, with $h>\ell$, we adapt the idea which leads to the point $(\alpha^*,\beta^*)$ from~\eqref{alpha*beta*}.\medskip\

Let $\alpha^{\sharp}\in \left[ 0,\tfrac{1}{\sqrt{5}}\right]$ satisfy the equation
\begin{equation}\label{alphaSharp}
\alpha = f(\alpha ,\beta ),\quad \text{where $\beta = -c_{\ell}^2\alpha + c_{\ell}$}.
\end{equation}
Since $f(\alpha,-c_{\ell}^2\alpha+c_{\ell}) = (1-c_{\ell})^2\alpha + 1-c_{\ell}$, we see that finding a solution to~\eqref{alphaSharp} is equivalent to solving
$$
\alpha = (1-c_{\ell})^2\alpha + 1-c_{\ell},
$$
and we find that 
\begin{equation}\label{alphaSharp2}
\alpha^{\sharp} = \frac{1-c_{\ell}}{c_{\ell}(2-c_{\ell})}.
\end{equation}
For $\alpha < \alpha^{\sharp}$ and $\beta_{\alpha} = -c_{\ell}^2\alpha+c_{\ell}$, we have $\alpha < f(\alpha,\beta_{\alpha})$, while for $\alpha > \alpha^{\sharp}$ and $\beta_{\alpha} = -c_{\ell}^2\alpha+c_{\ell}$, we have $\alpha > f(\alpha,\beta_{\alpha})$. A trivial but tedious calculation, using cross-multiplication and the ``inverse'' of~\eqref{cell2} (i.e., $c_{\ell-2}=\tfrac{2c_{\ell}-1}{1-c_{\ell}}$), yields
\begin{equation}\label{alphasharp}
\frac{1-c_{\ell}}{c_{\ell}(2-c_{\ell})} < \frac{g}{1+gc_{\ell-2}}.
\end{equation}
Recall that $ \frac{g}{1+gc_{\ell-2}}$ is the first coordinate of the top left vertex of $\Psi ( [c_{\ell},c_{\ell-2}]\times [g,c_{\ell}])$, and therefore the smallest value of $\alpha$ in $\Psi ( [c_{\ell},c_{\ell-2}]\times [g,c_{\ell}])$. We have the following result.

\begin{proposition}\label{prop:ellhoddh>ell}
For each $\ell$, $h$ odd, $h>\ell$, and for $(\alpha,\beta)\in \Psi ([c_{\ell},c_{\ell-2}]\times [c_h,c_{h-2}])$,
$$
\min \left\{ \alpha,\beta,f(\alpha,\beta)\right\} \leq B_{\ell,h} = f(\Psi (c_{\ell},c_{h-2})) = \frac{(1-c_{\ell})(1+c_{h-2})}{1+c_{\ell}c_{h-2}}\leq \frac{1-c_{\ell}^2}{1+c_{\ell}^2} \uparrow \tfrac{1}{\sqrt{5}},\quad \text{whenever $\ell\to\infty$}.
$$
Equivalently, for each $\ell$, $h$ odd, $h>\ell$, and for $(t_n,v_n)\in [c_{\ell},c_{\ell-2}]\times [c_h,c_{h-2}]$,
$$
\min \left\{ \Theta_{n-1},\Theta_n, \Theta_{n+1}\right\} \leq B_{\ell,h} = f(\Psi (c_{\ell},c_{h-2})) = \frac{(1-c_{\ell})(1+c_{h-2})}{1+c_{\ell}c_{h-2}}\leq \frac{1-c_{\ell}^2}{1+c_{\ell}^2} \uparrow \tfrac{1}{\sqrt{5}},\quad \text{whenever $\ell\to\infty$}.
$$
These constants are best possible. I.e., they cannot be replaced by smaller ones.
\end{proposition}

\begin{proof}
Let $(\alpha,\beta)\in\Psi ([c_{\ell},c_{\ell-2}]\times [c_h,c_{h-2}])$, where $\ell$, $h$ are \emph{odd}, and $h>\ell$. Now
$$
\frac{(1-c_{\ell})(1+c_{h-2})}{1+c_{\ell}c_{h-2}}\leq \frac{1-c_{\ell}^2}{1+c_{\ell}^2} < \tfrac{1}{\sqrt{5}}
$$
is an immediate consequence of the fact that if $h$, $\ell$ \emph{odd}, $\ell$ fixed and $h\to \infty$, we have $c_h\downarrow g$, so $(\Psi (c_{\ell},c_{h-2}))_{\text{$h>\ell$, $h$ odd}}$ is a sequence on the line $\beta = -c_{\ell}^2\alpha+c_{\ell}$ of which the first coordinates decrease, and with limit $\Psi (c_{\ell},g)$. So in each $\Psi ([c_{\ell},c_{\ell-2}]\times [c_h,c_{h-2}])$, $f$ takes its maximum in $\Psi (c_{\ell},c_{h-2})$ with value $\frac{(1-c_{\ell})(1+c_{h-2})}{1+c_{\ell}c_{h-2}}$, and due to Lemma~\ref{valuesfon''horizontallines''} these values are smaller than $f(\Psi (c_{\ell},c_{\ell}))=\tfrac{1-c_{\ell}^2}{1+c_{\ell}^2}\uparrow \tfrac{1}{\sqrt{5}}$, as $\ell\to\infty$.\medskip\

Next we show that for $(\alpha,\beta)\in\Psi ([c_{\ell},c_{\ell-2}]\times [c_h,c_{h-2}])$ we have
$$
\min \{ \alpha,\beta,f(\alpha,\beta)\} \leq f(\Psi (c_{\ell},c_{h-2})) = \frac{(1-c_{\ell})(1+c_{h-2})}{1+c_{\ell}c_{h-2}}.
$$
As $\Psi ([c_{\ell},c_{\ell-2}]\times [c_h,c_{h-2}])$ is ``above'' the diagonal $\beta =\alpha$ (it ``touches'' the diagonal in $\Psi (c_{\ell-2},c_{\ell-2})$ when $h-2=\ell$),  we immediately have $\alpha \leq \beta$, so the Borel constant $B_{\ell,h}$ for which $\min \{ \alpha,\beta,f(\alpha,\beta)\} \leq B_{\ell,h}$ cannot be determined by the value of $\beta$. In addition, this Borel constant $B_{\ell,h}$ cannot be determined by the value of $\alpha$. To see this, let $(\alpha,\beta)\in\Psi ([c_{\ell},c_{\ell-2}]\times [c_h,c_{h-2}])$ and set $\beta_{\alpha}=-c_{\ell}^2\alpha+c_{\ell}$. Then by~\eqref{alphasharp} we have $\alpha>\alpha^{\sharp}$, and~\eqref{derivatesarenegative} gives that
$$
f(\alpha,\beta)\leq f(\alpha,\beta_{\alpha})<\alpha \leq \beta .
$$
It follows that the constant $B_{\ell,h}$ for which $\min \{ \alpha,\beta,f(\alpha,\beta)\}\leq B_{\ell,h}$ for $(\alpha,\beta)\in\Psi ([c_{\ell},c_{\ell-2}]\times [c_h,c_{h-2}])$ is determined by the maximum value of $f(\alpha,\beta)$ in $\Psi ([c_{\ell},c_{\ell-2}]\times [c_h,c_{h-2}])$, which is $\tfrac{(1-c_{\ell})(1+c_{h-2})}{1+c_{\ell}c_{h-2}}$.
\end{proof}

\subsection{$(\ell,h)$ is (\emph{odd}, \emph{even})}\label{OddEven} Now we consider that part of $V_{\ell}^{\prime}$ not considered in the previous subsection; we have
$$
\Psi ([c_{\ell},c_{\ell-2}]\times [\tfrac{1}{2},g]) = \bigcup_{\text{$h\geq 2$ even}} \Psi ([c_{\ell},c_{\ell-2}]\times [c_{h-2},c_h]).
$$
Although $\alpha^{\sharp}$ played an important role in part ($iii$) of the previous subsection, Subsection~\ref{OddOdd}, its true ``significance'' will now be used. We first state (and prove) a lemma involving $\alpha^{\sharp}$, and then show why this lemma is instrumental in understanding this case. Recall from~\eqref{alphaSharp2} that $\alpha^{\sharp} = \tfrac{1-c_{\ell}}{c_{\ell}(2-c_{\ell})}$.

\begin{lemma}\label{lem:alphasharp1} For $\ell\in\N$ we have
\begin{equation}\label{eq::alphasharp1}
\alpha^{\sharp} = \frac{c_{\ell-1}}{1+c_{\ell}c_{\ell-1}} = \frac{c_{\ell+1}}{1+c_{\ell+1}c_{\ell-2}}.
\end{equation}
\end{lemma}

\begin{proof}
From definition~\eqref{cell} of the numbers $c_{\ell}$ and the definition of the Fibonacci numbers $(F_n)_{n\geq -1}$, we immediately have $c_{\ell}=\frac{1}{1+c_{\ell-1}}$; viz.
$$
c_{\ell} = \frac{F_{\ell+1}}{F_{\ell+2}} = \frac{F_{\ell+1}}{F_{\ell}+F_{\ell+1}} = \frac{1}{1+\displaystyle \frac{F_{\ell}}{F_{\ell+1}}} = \frac{1}{1+c_{\ell-1}}.
$$
But then we have $c_{\ell+1}=\frac{1}{1+c_{\ell}}$, and we find
$$
\alpha^{\sharp} = \frac{1-c_{\ell}}{c_{\ell}(2-c_{\ell})} = \frac{1-\frac{1}{1+c_{\ell-1}}}{\frac{1}{1+c_{\ell-1}}(2-\frac{1}{1+c_{\ell-1}})}
= \frac{c_{\ell-1}(1+c_{\ell-1})}{2c_{\ell-1}+1},
$$
and that
$$
\frac{c_{\ell-1}}{1+c_{\ell}c_{\ell-1}} = \frac{c_{\ell-1}}{1+c_{\ell-1}\cdot \frac{1}{1+c_{\ell-1}}} = \frac{c_{\ell-1}(1+c_{\ell-1})}{2c_{\ell-1}+1}.
$$
Furthermore, as~\eqref{cell2} yields that $c_{\ell-2} = \tfrac{2c_{\ell}-1}{1-c_{\ell}}$,
$$
\frac{c_{\ell+1}}{1+c_{\ell+1}c_{\ell-2}} = \frac{\frac{1}{1+c_{\ell}}}{1+\tfrac{1}{1+c_{\ell}}\cdot \tfrac{2c_{\ell}-1}{1-c_{\ell}}} = \frac{1}{1+c_{\ell} + \frac{2c_{\ell}-1}{1-c_{\ell}}} = \frac{1-c_{\ell}}{1-c_{\ell}^2+2c_{\ell}-1} = \frac{1-c_{\ell}}{c_{\ell}(2-c_{\ell})} = \alpha^{\sharp}.
$$
From these three results, the lemma follows.
\end{proof}

Hence, for $(\ell,h)$ is (\textit{odd}, \textit{even}) we look at the lower right rectangle $[g,1]\times [\tfrac{1}{2},g]$ \textit{below} the diagonal in the domain $[\frac{1}{2},1]\times [\tfrac{1}{2},1]$. Since this lower right rectangle is below the diagonal, $\Psi$ maps this onto a quadrilateral above the diagonal. Thus, $\alpha<\beta$ in the image and hence for $(\alpha,\beta)$ in this quadrilateral we have $\min \{ \alpha,\beta,f(\alpha,\beta)\} = \min \{ \alpha,f(\alpha,\beta)\}$.\smallskip

Let $(\alpha,\beta)\in\Psi ([c_{\ell},c_{\ell-2}]\times [c_h,c_{h-2}])$, where $\ell$, $h$ are \emph{odd} and \emph{even}, respectively. For fixed $\ell$ \emph{odd}, a case distinction will made for the cases $h\geq\ell+3$, $h\leq \ell-1$, and $h=\ell+1$. 
\newline
\newline
For $h=\ell+1$, the map $\Psi$ sends the points $A:\, (c_\ell,c_{\ell-1})$, $B:\, (c_{\ell-2},c_{\ell-1})$, $C:\, (c_{\ell-2},c_{\ell+1})$ and $D:\, (c_\ell, c_{\ell+1})$, to, respectively
$$
\Psi(A)=\left( \frac{c_{\ell-1}}{1+c_{\ell}c_{\ell-1}},\frac{c_{\ell}}{1+c_{\ell}c_{\ell-1}}\right),\qquad \Psi(B)=\left( \frac{c_{\ell-1}}{1+c_{\ell-2}c_{\ell-1}},\frac{c_{\ell-2}}{1+c_{\ell-2}c_{\ell-1}}\right),\, 
$$
and
$$
\Psi(C)=\left( \frac{c_{\ell+1}}{1+c_{\ell-2}c_{\ell+1}},\frac{c_{\ell-2}}{1+c_{\ell-2}c_{\ell+1}}\right),\qquad\Psi(D)=\left( \frac{c_{\ell+1}}{1+c_{\ell}c_{\ell+1}},\frac{c_{\ell}}{1+c_{\ell}c_{\ell+1}}\right).
$$
These are the bottom-left, top-left, top-right and bottom-right vertices, resp., in the image which is a quadrilateral.
From Lemma~\ref{lem:alphasharp1} it follows that the first coordinate of the top-right vertex $\Psi (C)$ and the first coordinate of the bottom-left vertex $\Psi (A)$ are equal to $\alpha^\sharp$. This means that the vertical line $\alpha = \alpha^\sharp$ divides the quadrilateral into two parts (actually two triangles). We first investigate the left part/left triangle. For $(\alpha,\beta)$ in this part it holds that $\alpha \le \alpha^\sharp$. As $\alpha^{\sharp}$ is the solution of $\alpha = f(\alpha, \beta_{\alpha})$, and $\Psi (A)$ is the bottom-right point of this left triangle, and because the first coordinate of $\Psi (A)$ is $\alpha^{\sharp}$, which is equal to $f(\Psi (A))$, it follows from~Lemma~\ref{valuesfon''verticallines''}  and~\eqref{derivatesarenegative} and that on this left part always $f(\alpha,\beta)\leq f(\Psi (A))=\alpha^{\sharp}$. Thus, the minimum of $\alpha$, $\beta$, and $f(\alpha,\beta)$ is for the points $(\alpha,\beta)$ in the left triangle always smaller than or equal to $\alpha^{\sharp}$.\smallskip\

\noindent 
For $(\alpha,\beta)$ in the right-hand part it holds that $\alpha \ge \alpha^\sharp$. For each $\alpha$ it holds that $f(\alpha,\beta_\alpha) \leq \alpha$. Since $f(\alpha,\beta) \leq f(\alpha,\beta_\alpha)\leq f(\Psi(D))=f(c_{\ell},c_{\ell+1})$, it follows that for $(\alpha,\beta)$ in this right-hand part
$$
\min\left\{\alpha,\beta,f(\alpha,\beta)\right\} \le \max \left\{ f(\alpha,\beta)\right\} = B_{\ell,h} := f(\Psi(D)) = f(c_{\ell},c_{\ell+1}).
$$
As the first coordinate of $\Psi (A)$ is $\alpha^{\sharp}$, and $\alpha^{\sharp}=f(\Psi(A))$, the results can be combined for the two parts as follows.
\begin{proposition}\label{prop:ellhoddh=ell+1}
For each $\ell$ odd, $h=\ell+1$, and for $(\alpha,\beta)\in \Psi ([c_{\ell},c_{\ell-2}]\times [c_{\ell-1},c_{\ell+1}])$, 
$$
\min \left\{ \alpha,\beta,f(\alpha,\beta)\right\} \leq B_{\ell,\ell+1} = f(\Psi (c_{\ell},c_{\ell+1})) = \frac{(1-c_\ell)(1+c_{\ell+1})}{c_{\ell} c_{\ell+1}+1} < \tfrac{1}{\sqrt{5}}.
$$
Equivalently, for each $\ell$ odd, $h=\ell+1$, and for $(t_n,v_n)\in [c_{\ell},c_{\ell-2}]\times [c_{\ell-1},c_{\ell+1}]$, 
$$
\min \left\{ \Theta_{n-1},\Theta_n, \Theta_{n+1}\right\} \leq B_{\ell,\ell+1} = f(\Psi (c_{\ell},c_{\ell+1})) = \frac{(1-c_\ell)(1+c_{\ell+1})}{c_{\ell} c_{\ell+1}+1}< \tfrac{1}{\sqrt{5}}.
$$
These constants are best possible. I.e., they cannot be replaced by smaller ones.
\end{proposition}
Now consider the case $h\geq\ell+3$. On the line $\beta = -c_{\ell}^2\alpha+c_{\ell}$ it holds that $\alpha>\alpha^{\sharp}$. 
For such $\alpha$, it holds that $\alpha > f(\alpha,\beta_\alpha)$. For each $(\alpha,\beta)\in \Psi([c_{\ell},c_{\ell-2}]\times[c_{h-2},c_h])$ with $h\geq\ell+3$ it holds that $f(\alpha,\beta) \le f(\alpha,\beta_\alpha) < \alpha$. These $\alpha's,\beta_\alpha's$ do not have to be in the same quadrilateral ($\Psi([c_{\ell},c_{\ell-2}]\times[c_{h-2},c_h])$), but they have to be in the region which $\Psi$ maps the case 
$$
h\geq\ell+1
$$
onto. This gives the following\footnote{Obviously, if we would only change $h\geq \ell+3$ into $h\geq \ell+1$ in Proposition~\ref{prop:ellhoddheven>ell+3}, this new proposition would contain Proposition~\ref{prop:ellhoddh=ell+1} as a special case. We decided to keep these two cases ($h=\ell+1$ and $h\geq \ell+3$) separate, for greater clarity of exposition.} result.
\begin{proposition}\label{prop:ellhoddheven>ell+3}
For each $\ell$ odd, $h\geq \ell+3$ even, and for $(\alpha,\beta)\in \Psi ([c_{\ell},c_{\ell-2}]\times [c_{h-2},c_h])$, 
$$
\min \left\{ \alpha,\beta,f(\alpha,\beta)\right\} \leq B_{\ell,h} = f(\Psi (c_{\ell},c_h)) = \frac{(1-c_\ell)(1+c_h)}{1+c_\ell c_h} < \tfrac{1}{\sqrt{5}}.
$$
Equivalently, for each $\ell$ odd, $h=\ell+1$, and for $(t_n,v_n)\in [c_{\ell},c_{\ell-2}]\times [c_{\ell-1},c_{\ell+1}]$, 
$$
\min \left\{ \Theta_{n-1},\Theta_n, \Theta_{n+1}\right\} \leq B_{\ell,h} = f(\Psi (c_{\ell},c_h))= \frac{(1-c_\ell)(1+c_h)}{1+c_\ell c_h} < \tfrac{1}{\sqrt{5}}.
$$
These constants are best possible. I.e., they cannot be replaced by smaller ones.
\end{proposition}

\begin{remark}\label{specialcaseRHS}
Although the proofs of Propositions~\ref{prop:ellhoddh=ell+1} and~\ref{prop:ellhoddheven>ell+3} yields this, it is not immediately clear why for $\ell$ odd and $h\geq \ell+1$ even, one has
\begin{equation}\label{unusualRHS}
\frac{(1-c_\ell)(1+c_h)}{1+c_\ell c_h} < \tfrac{1}{\sqrt{5}}.
\end{equation}
In order to show that~\eqref{unusualRHS} holds, firstly define for $\ell$ odd and fixed (so the numbers $c_{\ell}$ and $1-c_{\ell}$ are constants) the function $h$ as
$$
h(x) = \frac{(1-c_{\ell})(1+x)}{1+c_{\ell}x},\quad \text{for $x\in [0,1]$}.
$$
Then 
$$
h^{\prime}(x) = \frac{(1-c_{\ell})^2}{(1+c_{\ell}x)^2} > 0,
$$
for $x\in [0,1]$. So $h$ is an increasing function, and as $c_h<g$ (since $h$ is \emph{even}), it can be seen that for each $\ell$ odd, fixed, and each $h\geq \ell+1$ even,
$$
\frac{(1-c_\ell)(1+c_h)}{1+c_\ell c_h} = h(c_h) < \lim_{\text{$h\to\infty$, $h$ odd}} h(c_h) = h(g) = \frac{(1-c_{\ell})(1+g)}{1+c_{\ell}g}.
$$
Now consider the function $k$ as (and recall that $g+1=G$),
$$
k(x) = \frac{G(1-x)}{1+gx},\quad  \text{for $x\in [0,1]$}.
$$
Then
$$
k^{\prime}(x) = G\cdot \frac{-(1+gx) - g(1-x)}{(1+gx)^2} < 0,
$$
for $x\in [0,1]$. So $k$ is monotonically decreasing, and as $c_{\ell}>g$ (since $\ell$ is \emph{odd}), and $c_{\ell}\downarrow g$, as $\ell\to\infty$, it holds that
$$
k(c_{\ell}) < \lim_{\text{$\ell\to\infty$, $\ell$ odd}} k(c_{\ell}) = \frac{1-g^2}{1+g^2}=\frac{1}{\sqrt{5}}.
$$
So,
$$
\frac{(1-c_\ell)(1+c_h)}{1+c_\ell c_h} = h(c_h) < h(g) = k(c_{\ell}) < k(g)=\frac{1}{\sqrt{5}}.
$$\hfill$\triangle$
\end{remark}

\noindent
Lastly, let $h\leq\ell-1$. Define the map $\pi_1:\R^2\mapsto\R$ as the projection map in the first coordinate. Now for $h\leq \ell-1$, $\ell$ odd and $h$ even, it holds that $\alpha<\alpha^{\sharp}$, hence $\alpha<f(\alpha,\beta_\alpha)$. So, as always $\beta>\alpha$ on $\Psi ([c_{\ell},c_{\ell-2}]\times [c_{h-1},c_h])$, it can be seen that the minimum $\min\left\{ \alpha,\beta,f(\alpha,\beta)\right\}$ is smaller than or equal to the largest $\alpha$ for points $(\alpha,\beta)\in \Psi ([c_{\ell},c_{\ell-2}]\times [c_{h-1},c_h])$. Thus the following result obtained.

\begin{proposition}\label{prop:ellhoddheven<ell-1}
For each $\ell$ odd, $h\leq \ell-1$ even, and for $(\alpha,\beta)\in \Psi ([c_{\ell},c_{\ell-2}]\times [c_{h-2},c_h])$, 
$$
\min \left\{ \alpha,\beta,f(\alpha,\beta)\right\} \leq B_{\ell,h} = \pi_1(\Psi (c_{\ell},c_h))  = \frac{c_h}{1+c_{\ell}c_h} < \tfrac{1}{\sqrt{5}}.
$$
Equivalently, for each $\ell$ odd, $h\leq\ell-1$ even, and for $(t_n,v_n)\in [c_{\ell},c_{\ell-2}]\times [c_{h-2},c_h]$, 
$$
\min \left\{ \Theta_{n-1},\Theta_n, \Theta_{n+1}\right\} \leq B_{\ell,h} = \pi_1(\Psi (c_{\ell},c_h)) = \frac{c_h}{1+c_{\ell}c_h} < \tfrac{1}{\sqrt{5}}.
$$
These constants are best possible. I.e., they cannot be replaced by smaller ones.
\end{proposition}

\subsection{The case $(\ell,h)$ is (\emph{even}, \emph{odd})}\label{EvenOdd2}
Now consider $\ell$ and $h$ to be \textit{even} and \textit{odd}, respectively. Now $[c_{\ell-2},c_{\ell}]\times [c_h,c_{h-1}]$ is a subset of the top left rectangle $[\frac{1}{2},g]\times [g,1]$ in the domain $[\frac{1}{2},1]\times [\tfrac{1}{2},1]$. Since this last rectangle is above the diagonal, $\Psi$ maps this onto a quadrilateral below the diagonal. Thus, $\alpha>\beta$ for points $(\alpha,\beta)$ in this image, and we have $\min \left\{\alpha,\beta,f(\alpha,\beta)\right\} = \min\left\{ \beta,f(\alpha,\beta)\right\}$. Define
$$
A^* = \left(\tfrac{1}{2},g\right), \quad B^* = (g,g), \quad C^* = (g,1), \quad \text{and $D^* = \left(\tfrac{1}{2},1\right)$},
$$
then an easy calculation yields
$$
\Psi(A^*) = \left( \frac{2g}{2+g},\frac{1}{2+g}\right) = (2g^3,g^2),\quad \Psi(B^*) = \left( \frac{1}{\sqrt{5}},\frac{1}{\sqrt{5}}\right),\quad \Psi(C^*) = \left( g,g^2\right),\quad \text{and $\Psi(D^*) = \left( \tfrac{2}{3},\tfrac{1}{3} \right)$}.
$$
Thus, for all $(\alpha,\beta) \in\Psi{([\frac{1}{2},g]\times [g,1]})$ it is seen that $\alpha \geq \frac{1}{\sqrt{5}}$, and $\beta \leq \frac{1}{\sqrt{5}}$. 
Moreover, for these points $(\alpha,\beta)$ it follows from Lemmas~\ref{valuesfon''horizontallines''} and~\ref{valuesfon''verticallines''} that $f(\alpha,\beta) \geq f(\frac{1}{\sqrt{5}},\frac{1}{\sqrt{5}}) = \frac{1}{\sqrt{5}}$. Hence, for $(\alpha,\beta)\in \Psi([c_{\ell-2},c_{\ell}]\times [c_{h},c_{h-2}])\subsetneq \Psi ([\frac{1}{2},g]\times [g,1])$, we see that $\min\left\{ \alpha, \beta, f(\alpha,\beta)\right\}$ for such points $(\alpha,\beta)$ is always less than or equal to the largest value $\beta$ can achieve on $\Psi([c_{\ell-2},c_{\ell}]\times [c_{h},c_{h-2}])$, which is the second coordinate of $\Psi(B)$, where $B=(c_{\ell},c_h)$ (recall that $\ell$ is even). The following result is obtained.
\begin{proposition}\label{prop:ellevenhodd}
For each $\ell$ even, $h$ odd, and for $(\alpha,\beta)\in \Psi ([c_{\ell-2},c_{\ell}]\times [c_h,c_{h-2}])$, we have
$$
\min \left\{ \alpha,\beta,f(\alpha,\beta)\right\} \leq B_{\ell,h} = \frac{c_{\ell}}{1+c_{\ell}c_h} < \tfrac{1}{\sqrt{5}}.
$$
Equivalently, for each $\ell$ even, $h$ odd, and for $(t_n,v_n)\in [c_{\ell-2},c_{\ell}]\times [c_h,c_{h-2}]$, it holds that
$$
\min \left\{ \Theta_{n-1},\Theta_n,\Theta_{n+1}\right\} \leq B_{\ell,h} = \frac{c_{\ell}}{1+c_{\ell}c_h} < \tfrac{1}{\sqrt{5}}.
$$
These constants are best possible. I.e., they cannot be replaced by smaller ones.
\end{proposition}

\subsection{The case $(\ell,h)$ is (\emph{even}, \emph{even})}\label{EvenEven}
Now for $\ell$ and $h$ both \textit{even}, rectangles of the form $[c_{\ell-2},c_{\ell}]\times [c_{h-2},c_h]$ inside the bottom left square in the domain $[\frac{1}{2},g]\times[\frac{1}{2},g]$ are being investigated. The vertices of the bottom left square square are
$$
A^*\mathrel{\mathop:} = \left( \tfrac{1}{2},\tfrac{1}{2}\right),\quad  
B^*\mathrel{\mathop:}= \left( g,\tfrac{1}{2}\right),\quad   
C^*\mathrel{\mathop:}=(g,g),\quad   
D^*\mathrel{\mathop:} = \left( \tfrac{1}{2},g\right)  .
$$
Now $\Psi$ maps this square to the quadrilateral with vertices:
$$
\Psi(A^*) = \left( \tfrac{2}{2}, \tfrac{2}{5}\right), \quad \Psi(B^*) =  \left( g^2,2g^3\right), \quad   
\Psi(C^*) = \left( \tfrac{1}{\sqrt{5}},\tfrac{1}{\sqrt{5}}\right), \quad \Psi(D^*) = \left(  2g^3,g^2\right). 
$$
Again, we divide this region into three parts; $\ell$ and $h$ such that the rectangles are on the diagonal $\alpha=\beta$, $\ell$ and $h$ such that the rectangles are above the diagonal, and $\ell$ and $h$ such that the rectangles are below the diagonal. 
\newline
\newline
Case ($i$): 
Starting with $\ell=h\geq 2$ \emph{even}, hence where the rectangles are on the diagonal. Now for $A$, $B$, $C$ and $D$ the vertices of the bottom-left, bottom-right, upper-right and upper-left, resp.\ corresponding to an arbitrary rectangle $[c_{\ell-2},c_{\ell}]\times [c_{\ell-2},c_{\ell}]$ on the diagonal, 
$$
A=(c_{\ell-2},c_{\ell-2}), \quad  
B=(c_{\ell},c_{\ell-2}), \quad  
C=(c_{\ell},c_{\ell}), \quad  
D=(c_{\ell-2},c_{\ell}).
$$
As the first coordinate (and therefore also the second coordinate) of $\Psi(C)$ is smaller than $\tfrac{1}{\sqrt{5}}$, it follows from~\eqref{derivatesarenegative} that $f(\Psi(C)) > f(\Psi(C^*))=\tfrac{1}{\sqrt{5}}$; see~\eqref{eveneven}.

The top-right vertex divides the quadrilateral into three parts. The upper right-vertex corresponds to $\Psi(C)= \Psi\left( c_{\ell},c_{\ell}\right) = \left( \frac{c_\ell}{1+c^2_\ell}, \frac{c_\ell}{1+c^2_\ell}\right)$. The horizontal line, $\beta = \frac{c_\ell}{1+c^2_\ell}$, and the vertical line, $\alpha = \frac{c_\ell}{1+c^2_\ell}$, corresponding to the coordinates of this vertex divide the area into three regions. One region is bounded by both the vertical and horizontal lines, the second region is bounded from below by the horizontal line and the last region is bounded on the left by the vertical line. For all $(\alpha,\beta)$ in the first region we have $\alpha, \beta\le\frac{c_\ell}{1+c^2_\ell}$. For all $(\alpha,\beta)$ on the second region we have $\alpha<\beta$ since this is above the diagonal, and due to the geometry of this part it follows that $\alpha\leq \frac{c_\ell}{1+c^2_\ell}$ and $\beta\ge\frac{c_\ell}{1+c^2_\ell}$. Hence $\min\left\{\alpha,\beta\right\}\le\frac{c_\ell}{1+c^2_\ell}$. Lastly for all $(\alpha,\beta)$ on the last region it holds that $\alpha\ge\frac{c_\ell}{1+c^2_\ell}$ and $\beta\le\frac{c_\ell}{1+c^2_\ell}$. Thus $\min\left\{\alpha,\beta\right\}\le\frac{c_\ell}{1+c^2_\ell}$. Note that $\frac{c_\ell}{1+c^2_\ell}<\tfrac{1}{\sqrt{5}}$.
\newline
\newline
This means that $\min\left\{\alpha,\beta\right\}\le \frac{c_\ell}{1+c^2_\ell}$. Now in order for $\min\left\{ \alpha,\beta\,f(\alpha,\beta) \right\}\le \frac{c_\ell}{1+c^2_\ell}$ to hold (of course this holds, but $\frac{c_\ell}{1+c^2_\ell}$ might be replaced by an even smaller constant), it is enough to show that $f(C) > \frac{c_\ell}{1+c^2_\ell} $. This is indeed the case since:
$$
f(\Psi(C))=f\Big( \frac{c_\ell}{1+c^2_\ell}, \frac{c_\ell}{1+c^2_\ell}\Big) = \sqrt{1-4\frac{c^2_\ell}{(1+c^2_\ell)^2}} \quad
=\sqrt{\frac{(c^2_\ell-1)^2}{(c^2_\ell+1)^2}} = \frac{1-c^2_\ell}{1+c^2_\ell}.
$$
For $\ell$ \emph{even}, we have that
\begin{equation}\label{eveneven}
\frac{1-c^2_\ell}{c^2_\ell+1} > \frac{1}{\sqrt{5}}.
\end{equation}
To see that~\eqref{eveneven} is holds, it suffices to consider the function $s:[0,1]\to\R$, defined by $s(x)=\frac{1-x}{1+x}$. Then $s^{\prime}(x)<0$ for all $x\in [0,1]$, and as $s(g^2)=\frac{1}{\sqrt{5}}$, from $c_{\ell} < g$ for $\ell$ \emph{even}, it immediately follows that
$$
\frac{1-c^2_\ell}{c^2_\ell+1} = s(c_{\ell}^2) > s(g^2) = \frac{1}{\sqrt{5}}.
$$

\begin{remark}{\rm
Note that the above statement is \emph{only} true if $\ell$ is even. For example, take $c_1=\tfrac{2}{3}$, then}
$$
\frac{1-c^2_\ell}{c^2_\ell+1} = \frac{1-\tfrac{4}{9}}{1+\tfrac{4}{9}} = \frac{5}{13} = 0.38\cdots \ll 0.447\cdots = \frac{1}{\sqrt{5}}.
$$
\hfill $\triangle$
\end{remark}
The following result holds.\smallskip\

\begin{proposition}\label{prop:ellhbothevenh=ell}
Let $\ell$ and $h$ both be even and $h=\ell$. Then for $(\alpha,\beta) \in \Psi([c_{\ell-2},c_\ell]\times[c_{\ell-2},c_{\ell}])$, we have
$$
\min\left\{\alpha,\beta,f(\alpha,\beta)\right\} \le B_{\ell,\ell} = \frac{c_{\ell}}{1+ c_{\ell}^2} < \frac{1}{\sqrt5}.
$$
Equivalently, for $\ell$ and $h$ both even, $h=\ell$, and for $(t_n,v_n) \in [c_{\ell-2},c_\ell]\times[c_{\ell-2},c_{\ell}]$, it holds that
$$
\min\left\{\Theta_{n-2},\Theta_n, \Theta_{n+1}\right\} \le B_{\ell,\ell} = \frac{c_{\ell}}{1+ c_{\ell}^2} < \frac{1}{\sqrt5}.
$$
These constants $B_{\ell,\ell}$ are best possible. I.e., they cannot be replaced by smaller ones.
\end{proposition}

\noindent
Case ($ii$): Now consider the case $h\ge \ell+2$, so where the rectangle $[c_{\ell-2},c_{\ell}]\times [c_{h-2},c_h]$ is above the diagonal, and $\Psi$ maps this rectangle below the diagonal in the image. Due to this, for all $(\alpha,\beta)\in \Psi ([c_{\ell-2},c_{\ell}]\times [c_{h-2},c_h])$ we have $\beta<\alpha$, and also (and this is due to Lemma~\ref{valuesfon''horizontallines''}, Lemma~\ref{valuesfon''verticallines''} and because of what was seen in part ($i$) that $f(\Psi (c_{\ell},c_{\ell})>\tfrac{1}{\sqrt{5}}$), that $f(\alpha,\beta) > \tfrac{1}{\sqrt{5}}$. As the second coordinate of $\Psi (c_{\ell},c_{h-2})= \left( \frac{c_{h-2}}{1+c_{h-2}c_{\ell}}, \frac{c_{\ell}}{1+c_{h-2}c_{\ell}}\right)$ is the largest $\beta$-value on $\Psi ([c_{\ell-2},c_{\ell}]\times [c_{h-2},c_h])$ (this is because $\Psi ([c_{\ell-2},c_{\ell}]\times [c_{h-2},c_h])$ is bounded by straight lines with negative slopes), it follows for $(\alpha,\beta)\in \Psi([c_{\ell-2},c_{\ell}]\times [c_{h-2},c_h])$ that
$$
\min\left\{ \alpha,\beta,f(\alpha,\beta)\right\}\leq B_{\ell,h} = \frac{c_{\ell}}{1+c_{h-2}c_{\ell}} = \pi_2(\Psi(B)).
$$
where $B=(c_{\ell},c_{h-2})$. Thus we obtained the following (partial) result\smallskip

\begin{proposition}\label{prop:ellhbothevenh>ell}
Let $\ell$ and $h$ both be even, and $h\geq\ell+2$. Then for $(\alpha,\beta) \in \Psi([c_{\ell-2},c_\ell]\times[c_{h-2},c_h])$, it holds that
$$
\min\left\{\alpha,\beta,f(\alpha,\beta)\right\} \le B_{\ell,h} = \frac{c_{\ell}}{1+c_{h-2}c_{\ell}} < \frac{1}{\sqrt5}.
$$
Equivalently, for $\ell$ and $h$ both even, $h\geq \ell+2$, and for $(t_n,v_n) \in [c_{\ell-2},c_\ell]\times[c_{h-2},c_h]$, it holds that
$$
\min\left\{\Theta_{n-2},\Theta_n, \Theta_{n+1}\right\} \le B_{\ell,h} = \frac{c_{\ell}}{1+c_{h-2}c_{\ell}} < \frac{1}{\sqrt5}.
$$
These constants $B_{\ell,h}$ are best possible. I.e., they cannot be replaced by smaller ones.
\end{proposition}

\noindent
Case ($iii$): Now for all $h\leq \ell-2$, hence where the rectangle $[c_{\ell-2},c_{\ell}]\times [c_{h-2},c_h]$ is below the diagonal, and $\Psi$ maps this rectangle above the diagonal in the image. Clearly, this case is void if $\ell=2$. Note that now $\Psi ([c_{\ell-2},c_{\ell}]\times [c_{h-2},c_h])$ is bounded from below by the line $\beta = -c_{\ell-2}^2\alpha+c_{\ell-2}$. Determining $\alpha^{\sharp}$ for this situation, the following equation must to be solved
$$
\alpha = f(\alpha,\beta_{\alpha}),
$$
where $\beta_{\alpha} = -c_{\ell-2}^2\alpha+c_{\ell-2}$. A calculation similar to the case discussed in~\eqref{alphaSharp} now yields a solution $\alpha^{\flat}$ to $\alpha = f(\alpha,\beta_{\alpha})$,
$$
\alpha^{\flat} = \frac{1-c_{\ell-2}}{c_{\ell-2}(2-c_{\ell-2})}.
$$
Again, as with $\alpha^{\sharp}$, it is seen that for $\alpha < \alpha^{\flat}$ it holds that $\alpha < f(\alpha,\beta_{\alpha})$, and conversely that if $\alpha > \alpha^{\flat}$ it holds that $\alpha > f(\alpha,\beta_{\alpha})$.\smallskip\

\noindent
Now consider the function $k : [\frac{1}{2},g]\to\R$ defined by
$$
k(x)=\frac{1-x}{x(2-x)}, \quad \text{for $x\in [\tfrac{1}{2},g]$}.
$$
Then $k'(x)=\frac{-x^2+2x-2}{(x(2-x))^2}<0$. So in $[\frac{1}{2},g]$ the function $k$ attains its maximum in $x=\frac{1}{2}$, this maximum being $k(\frac{1}{2})=\frac{2}{3}$. In addition $k$ attains its minimum in $x=g$, this minimal value being $k(g)=\frac{1}{\sqrt{5}}$. 
Since $\ell$ is even, it holds that $\frac{1}{2}\le c_\ell < g$ and therefore $\alpha^{\flat}=k(c_{\ell-2}) > \frac{1}{\sqrt{5}}$.
\newline
Furthermore, for $\ell$ and $h$ even, $0< h\le \ell -2$, it holds that $\Psi([c_{\ell-2},c_\ell]\times[c_{h-2},c_h])$ lies above the diagonal. This means that $\beta>\alpha$ for $(\alpha,\beta)\in\Psi([c_{\ell-2},c_\ell]\times[c_{h-2},c_h])$, and additionally that $\alpha<\frac{1}{\sqrt5}<\alpha^{\flat}$.\\
But then $\alpha<\alpha^\flat$, and it follows that on $\Psi([c_{\ell-2},c_\ell]\times[c_{h-2},c_h])$ it holds that $\alpha \le f(\alpha,\beta_\alpha)$, where $\beta_\alpha=-c_{\ell-2}^2\alpha + c_{\ell-2}$. So on $\Psi([c_{\ell-2},c_\ell]\times[c_{h-2},c_h])$ it is seen that $\alpha$, $\beta$ and $f(\alpha,\beta)$ are all smaller than, or equal to the first coordinate of the lower right-hand vertex of $\Psi([c_{\ell-2},c_\ell]\times[c_{h-2},c_h])$, which is
$$
\Psi(c_{\ell-2},c_h) = \left( \frac{c_h}{1+c_{\ell-2}c_h},\frac{c_{\ell-2}}{1+c_{\ell-2}c_h}\right).
$$

This yields the following proposition.\smallskip\

\begin{proposition}\label{prop:ellhbothevenh<ell}
Let $\ell$ and $h$ both be even and $0<h\le \ \ell -2$. Then for $(\alpha,\beta) \in \Psi([c_{\ell-2},c_\ell]\times[c_{h-2},c_h])$, it holds that
$$
\min\left\{\alpha,\beta,f(\alpha,\beta)\right\} \le B_{\ell,h} = \frac{c_h}{1+ c_{\ell-2}c_h} < \frac{1}{\sqrt5}.
$$
Equivalently, for $\ell$ and $h$ both even, $0<h\le \ \ell -2$, and for $(t_n,v_n) \in [c_{\ell-2},c_\ell]\times[c_{h-2},c_h]$, it holds that
$$
\min\left\{\Theta_{n-1},\Theta_n, \Theta_{n+1}\right\} \le B_{\ell,h} = \frac{c_h}{1+ c_{\ell-2}c_h} < \frac{1}{\sqrt5}.
$$
These constants $B_{\ell,h}$ are best possible. I.e., they cannot be replaced by smaller ones.
\end{proposition}

\subsection{Borel for the special case $t_n=g$}\label{section5}
In this section, we assume that there exists an $n\in\N\cup\{ 0\}$, such that $t_n=g$. As $g$ is a fixed point of the Gauss map $T$, we find $T^m(x)=g$, for all $m\geq n$. Because of this, in this section we will work on the natural extension $[0,1)\times [0,1]$, or even more precisely: on the vertical line segment $t=g$ and $0\leq v\leq 1$. We consider two cases. In the next subsection, we assume that $n\geq 1$ and $v_n\neq [0;1^n]$. The last case, where $t_0=0$ (and therefore $v_0=0$ and $v_n=[0;1^n]$, for $n\in\N$) is dealt with in Subsection~\ref{subsec:t-0=0}.

\subsubsection{The case $t_n=g$ and $v_n\neq [0;1^n]$, for some $n\in\N$}\label{subsec:t-n=g}
As $x=[0;a_1,a_2,\dots, a_n, a_{n+1},\dots]$ we have $a_m=1$ for all $m\geq n+1$, and $v_n=[0;a_n,a_{n-1},\dots,a_1]$.\medskip\

We first consider the case $a_n\neq 1$. Then $v_n\in (\frac{1}{a_n+1},\frac{1}{a_n}]\subsetneq [0,\frac{1}{2}]$, so from~\eqref{thetan} we find, as $t_n=g$,
$$
\Theta_n = \frac{g}{1+g\cdot v_n} \in \left[ \frac{2g}{2+g},g\right].
$$
It follows that $\Theta_n\geq \frac{2g}{2+g} = 2g^3 = 0.472135\cdots > \frac{1}{\sqrt{5}}$ (note that we can considerably refine this using the actual value of $a_n\neq 1$). From~\eqref{thetan-1} we find, as $t_n=g$,
$$
\Theta_{n-1} = \frac{v_n}{1+g\cdot v_n} \in [ 0,g^2],
$$
so $\Theta_{n-1}\leq g^2=0.381966\cdots$. Again, a much sharper result is possible if we use the actual value of $a_n\neq 1$.\medskip\

To determine the size of $\Theta_{n+1}$ we could use~\eqref{thetasA}. However, as $t_{n+1}=g$, and
$$
v_{n+1}=[0;1,a_n,\dots,a_1] = \frac{1}{1+v_n}\in [\tfrac{2}{3},1],
$$
we see that, again using~\eqref{thetan}, but with $n$ replaced by $n+1$,
$$
\Theta_{n+1} = \frac{g}{1+g\cdot v_{n+1}} = \frac{g(1+v_n)}{1+v_n+g} \in \left[ g^2, \frac{3g}{3+2g}\right] = [ 0.381966\cdots , 0.437694\cdots].
$$
Also, here much can be gained by using the value of $a_n\geq 2$. Thus, we find that in this case
$$
\min \{ \Theta_{n-1}, \Theta_n, \Theta_{n+1}\} = \Theta_{n-1}\leq g^2.
$$

Next, let $1\leq \ell < n$ be such that $v_n = [0;1^{\ell}, a_{n-\ell},\dots,a_1]$; so $a_{n-\ell}\geq 2$. We again have two subcases here: $\ell$ is \emph{even}, or $\ell$ is \emph{odd}. Note that we assumed that $\ell\neq n$. If $\ell =n$, we have that $x=g$, and this case we treat seperately in the next sub section.\smallskip\

In case $\ell$ is \emph{odd}, we have $g < c_{\ell}\leq v_n\leq c_{\ell-2}\leq 1$. In this case, it immediately follows that both $v_{n-1}<g$ and $v_{n+1}<g$, and as $\ell\geq 1$ we also see that $t_{n-1}=g$. But then it is easy to see that
$$
\Theta_n < \frac{1}{\sqrt{5}} < \Theta_{n+1} < \Theta_{n-1}.
$$
To see this, note that
$$
\frac{v}{1+g\cdot v} \downarrow \frac{1}{\sqrt{5}}\,\, \text{ and }\,\, \frac{g}{1+g\cdot v} \uparrow \frac{1}{\sqrt{5}},\quad \text{whenever $v\downarrow g$},
$$
while
$$
\frac{v}{1+g\cdot v} \uparrow \frac{1}{\sqrt{5}}\,\, \text{ and }\,\,\frac{g}{1+g\cdot v} \downarrow \frac{1}{\sqrt{5}},\quad \text{whenever $v\uparrow g$},
$$
and that these limits are increasing respectively and decreasing monotonically to $\tfrac{1}{\sqrt{5}}$. Now
$$
\Theta_{n-1}\in \Big[ \frac{c_{\ell}}{1+gc_{\ell}}, \frac{c_{\ell-2}}{1+gc_{\ell-2}}\Big),\quad
\Theta_n\in \Big( \frac{g}{1+g\cdot c_{\ell-2}}, \frac{g}{1+g\cdot c_{\ell}} \Big],
$$
so $\Theta_n < \frac{1}{\sqrt{5}} < \Theta_{n-1}$, as $\ell$ \emph{odd} implies that $g<c_{\ell}$. Since $v_{n+1}=\frac{1}{1+v_n}$ we see
$$
\Theta_{n+1} = \frac{g}{1+gv_{n+1}} = \frac{g(1+v_n)}{G+v_n}\in \Big[ \frac{g(1+c_{\ell})}{G+c_{\ell}}, \frac{g(1+c_{\ell-2})}{G+c_{\ell-2}}\Big),
$$
and due to~\eqref{cell2} we immediately find that $\Theta_{n+1}<\Theta_{n-1}$, as
\begin{equation}\label{RemarkableEquation}
\frac{c_{\ell}}{1+gc_{\ell}} = \frac{\frac{1+c_{\ell-2}}{2+c_{\ell-2}}}{1+g\cdot\frac{1+c_{\ell-2}}{2+c_{\ell-2}}} = \frac{1+c_{\ell-2}}{G^2+Gc_{\ell-2}} = \frac{g(1+c_{\ell-2})}{G+c_{\ell-2}}.
\end{equation}
Thus, in this case we find that
$$
\min \{ \Theta_{n-1}, \Theta_n, \Theta_{n+1}\} = \Theta_n\leq \frac{g}{1+g\cdot c_{\ell}} < \frac{1}{\sqrt{5}}<\Theta_{n+1}<\Theta_{n-1}.
$$

In case $\ell$ is \emph{even}, we have $\tfrac{1}{2}\leq c_{\ell-2} < v_n\leq c_{\ell} < g$. In this case, it immediately follows that both $v_{n-1}>g$ and $v_{n+1}>g$, and as $\ell\geq 2$ we also see that $t_{n-1}=g$. But then, similar to the case $\ell$ is \emph{odd}, and again using~\eqref{RemarkableEquation}, it is easy to see that
$$
\Theta_{n-1} < \Theta_{n+1} < \frac{1}{\sqrt{5}} < \Theta_n.
$$

Now the function $h: [0,1]\to\R$, defined by $h(x)=\frac{x}{1+gx}$ has derivative $h'(x)>0$, so
$$
\Theta_{n-1} = \frac{v_n}{1+g\cdot v_n} \in [h(c_{\ell-2}),h(c_{\ell})] =\left[  \frac{c_{\ell-2}}{1+gc_{\ell-2}}, \frac{c_{\ell}}{1+gc_{\ell}}\right],
$$
and we find in this case, where $\ell$ is \emph{even}, that
$$
\min \{ \Theta_{n-1}, \Theta_n, \Theta_{n+1}\} = \Theta_{n-1}\leq \frac{c_{\ell}}{1+gc_{\ell}} < \Theta_{n+1} < \tfrac{1}{\sqrt{5}} < \Theta_n.
$$

\subsubsection{The case $t_0=g$}\label{subsec:t-0=0}
In this final section we have $t_0=g$, and therefore $t_n=g$ for all $n\in\N\cup\{ 0\}$. In this case, we find a result along the lines of the results by Jaroslav Han\u{c}l and Radhakrishnan Nair (\cite{[HN]}). By definition $v_0=0$, and using the natural extension map $\mathcal{T}$ from~\eqref{naturalextensionmap}, we see that $(t_0,v_0)=(g,0)$,
$$
(t_1,v_1) = \mathcal{T}(g,0) = (g,1) = (g,c_{-1}),\,\, (t_2,v_2) = \mathcal{T}(g,1) = (g, \tfrac{1}{2})=(g,c_0),\,\, (t_3,v_3) = \mathcal{T}(g,\tfrac{1}{2}) = (g,\tfrac{2}{3}) = (g,c_1),
$$
and, in general, for $n\in\N$ we see by induction that
$$
(t_n,v_n) = (g,[0;1^n]) = (g,c_{n-2}).
$$
But then $\Theta_0=g$, and for $n\in\N$,
$
\Theta_n = \frac{g}{1+gc_{n-2}}.
$
As
$$
c_0=\tfrac{1}{2} < c_2=\tfrac{3}{5} < c_4=\tfrac{8}{13} < \cdots < g < \cdots < c_3 = \tfrac{5}{8} < c_1 = \tfrac{2}{3} < c_{-1} =1,
$$
it follows for $n\in\N$ that
$$
\min \{ \Theta_{n-1}, \Theta_n, \Theta_{n+1}\} = \begin{cases}
\Theta_n = {\displaystyle \frac{g}{1+g\cdot c_{n-2}}} < \frac{1}{\sqrt{5}} < \Theta_{n+1}<\Theta_{n-1}, & \text{if $n$ is \emph{odd}}\\
 & \\
\Theta_{n-1}  = {\displaystyle \frac{g}{1+g\cdot c_{n-3}}} < \Theta_{n+1} < \frac{1}{\sqrt{5}} < \Theta_n, & \text{if $n$ is \emph{even}}. 
\end{cases}
$$\medskip\

\section{Concluding remarks}\label{sec:remarks}
In the previous sections, $h\in\N\cup\{ 0\}$ was such, that $v_n=[0;1^h,a_{n-h},\dots,a_1]$ (if $h=n$, we have $v_n=[0;1^n]$). Here $v_n$ is the past of some real number $x\in [0,1)$ at time $n$. This gives a clear bound on $h$ (obviously $h\leq n$), and therefore $h$ can not be arbitrarily large (the presentation in the previous sections might suggest that), and therefore the bounds obtained in the various sub-sections are sharper than at first one might assume.\smallskip\

\noindent
A second remark is that in this paper the focus was on Borel's result, but it should be clear that the method used can also be applied to sharpen~\eqref{improvementBorel} by Bagemihl and McLaughlin from~\cite{[BMc]}. Probably, this improvement will come at the cost of a slightly more complicated case distinction. Also conceptually it is easier to work with the Borel bound $\tfrac{1}{\sqrt{5}}$ rather than with the more ``abstract'' Bagemihl and McLaughlin bound $\tfrac{1}{\sqrt{a^2+4}}$. We refrained from refining the result of Bagemihl and McLaughlin. However, if one would refine Bagemihl and McLaughlin's result~\eqref{improvementBorel}, it is very likely that a refinement of Tong's result~\eqref{Tong} will come ``for free'', as a spin-off of the method used.

\section*{Acknowledgement}
The second author  acknowledges the support of the Institute of Mathematics of the Polish Academy of Sciences (IMPAN)  for providing excellent working conditions and research travel funds. The third author acknowledges the hospitality of the Institute of Mathematics of the Polish Academy of Sciences. 
The second and third authors also thank the organizers of the workshop ``Uniform Distribution of Sequences,'' held in April 2025 at the Erwin Schr\"odinger International Institute for Mathematics and Physics (ESI) in Vienna, for their hospitality, where the initial discussions leading to this project took place.
 \smallskip\

\Addresses
\end{document}